\documentclass[12pt]{article}
\usepackage{layout,yfonts, graphicx, amsmath, amsthm, amssymb, euscript, enumerate, enumitem, bbold, bbm, mathrsfs,easyReview}

\usepackage{caption}
\usepackage{subcaption}

\usepackage[numbers]{natbib} %cambia el estilo citar la bibliografia
\numberwithin{equation}{section}
\usepackage{color}
\definecolor{webgreen}{rgb}{0,.5,0}
\definecolor{webbrown}{rgb}{0.0,0.0,0.0}
\definecolor{emphcolor}{rgb}{0.95,0.95,0.95}

\usepackage{hyperref}
\hypersetup{%
	colorlinks=true,
	linkcolor=webbrown,
	urlcolor=webbrown,
	filecolor=webbrown,
	citecolor=webbrown,
	breaklinks=true}
\ifpdf \hypersetup{pdftex,
	bookmarksopen=true,
	bookmarksnumbered=true
} \else \hypersetup{dvips} \fi

\usepackage{blindtext} %Para enumerar
\usepackage{scrextend}
\addtokomafont{labelinglabel}{\sffamily}

\usepackage[margin=1in]{geometry}
\theoremstyle{plain}
\newtheorem{theorem}{Theorem}[section]

\newtheorem{lemma}[theorem]{Lemma}

\newtheorem{remark}[theorem]{Remark}

\newtheoremstyle{hyp}{}{}{\itshape}{}{}{}{3pt}{}
\theoremstyle{hyp}

\DeclareMathAlphabet{\mathpzc}{OT1}{pzc}{m}{it}

\newcommand{\uno}{\mathbbm{1}}

\numberwithin{equation}{section}
\DeclareMathOperator{\expo}{e}

\newcommand{\eqdef}{\raisebox{0.4pt}{\ensuremath{:}}\hspace*{-1mm}=}

\newcommand {\R}{\mathbb{R}}

\newcommand {\A}{\mathcal{A}}

\newcommand {\E}{\mathbb{E}}

\newcommand{\diff}{{\rm d}}
\newcommand{\sell}{\mathpzc {s}}

\newcommand{\cont}{{\rm c}}

\newcommand{\blue}{\textcolor[rgb]{0.00,0.0,1.00}}

\newcommand{\Gb}{\mathbb{G}}
\newcommand{\Fb}{\mathbb{F}}

\title{\Large{\sc {A mixed singular/switching  control problem with terminal cost for   modulated  diffusion processes}}\footnote{\textbf{Funding:}  {This} study has been funded by the Russian Academic Excellence Project `5-100'. \blue{Additionally, M. Kelbert was supported by the RSF Grant number 23-21-00052}}}
\author{\large{\bf Mark Kelbert}\\ 
	\large{\bf Harold A. Moreno-Franco}\footnote{Corresponding author: hmoreno@hse.ru}\\
	\small{\it Department of Statistics and Data Analysis}\\ 
	\small{\it Laboratory of Stochastic Analysis and its Applications}\\ 
	\small{\it National Research University Higher School of Economics, Moscow, Russian Federation}}%\\
\date{}
\allowdisplaybreaks
\begin{document}

\begin{center}
	{\LARGE\bf  {Optimal Liquidation in a Defaultable Market}}\\
	\vspace{0.4cm}
		{\large {\bf Daniel Hern\'andez-Hern\'andez}\footnote{Department of Probability and Statistics, Centro de Investigaci\'on en Matem\'aticas A.C., Mexico. Email: dher@cimat.mx.},} {\large  {\bf Harold A. Moreno-Franco}}\footnote{Department of Statistics and Data Analysis, Laboratory of Stochastic Analysis and its Applications, National Research University Higher School of Economics, Federation of  Russia. Email: hmoreno@hse.ru.} {\large {\bf and Jos\'e-Luis P\'erez}}\footnote{Department of Probability and Statistics, Centro de Investigaci\'on en Matem\'aticas A.C., Mexico. Email: jluis.garmendia@cimat.mx.}
\end{center}

\begin{abstract}
	\noindent In this paper we address the problem of optimal liquidation of a large portfolio composed by securities exposed to default risk. The default time is described in terms of a Brownian motion representing the evolution of the  value of the firm, whose assets are available in the market for investors. Considering that selling a large number of assets has a significant impact in the price, and hence in the portfolio's value, the control problem involved to describe the optimal strategy to liquidate a large position is analyzed. Under suitable assumptions in the model, an explicit solution is given to the value function and a precise description of the optimal strategy is obtained. 
	
\end{abstract}

\section{Introduction}
In the classical problem of portfolio optimization it is considered that the investor has a utility function and a set of stocks  used to design an investment strategy which allows him to maximized the discounted value of his  expected  terminal wealth.  However, there are other scenarios in which an investor is interested in selling a large number of shares with the aim of obtaining the maximum expected profit. The problem studied in this paper falls within the latter category, which encompasses  those problems of optimal liquidation. We refer to the work of Alfonsi et al. \cite{AFS, AFS2} and Almgren and Chriss \cite{A-C-2,A-C-3}   for the description of the main elements that must be taken into account when modelling the dynamics involved in the market in this type of liquidation strategy; related with this problem we can also mention the works  by Engle and Ferstenberg \cite{EF}, Gatheral and Shied \cite{GS}, Huberman and Stanzi \cite{HS} and Prediou et al. \cite{PSS}.   

This class of problems naturally gives rise to optimal stochastic control problems. In particular, when diffusion-type models   are used to describe the state dynamics, the solution involves the analysis of HJB equations in order to characterize the value function and the optimal control; see, for instance, the contributions of Guo and Zervos \cite{ZG15} and Hern\'andez-Hern\'andez et. al.\cite{HMP19} on this connection.    Here we consider a market model that leads to default by the firm issuing the shares. This default depends on the company's financial statements, and in this paper we consider the existence of a critical level, described as a barrier in a functional of the stochastic process that describes them, and once this is reached, default by the firm is declared.  As we might have in a high frequency market, complete observation of the asset prices is  assumed. When the liquidity is an issue to execute a position there are works by Lokka \cite{L}, Obizhaeva and Wang \cite{OW}, Shield and Schoneborn \cite{SS} and Schied et al. \cite{SST} addressing those problems. In this work we assume that there is a multiplicative impact in the price, but there are other ways to medel it, as in Bertsimas  and Lo \cite{BL},  Gatheral et al. \cite{GSS} or He and Mamayski \cite{HM}.

Here, we follow the intensity approach to model the default time, based  in the the work of Elliott et.al \cite{E-J-Y}. It consists in describing the default process (a single jump process),  defined in terms of a stopping time, through an intensity process, so that the compensated process is a martingale. Having this extra uncertainty in mind, the problem we are interested consists in finding the optimal execution policy, when we have a large position in the asset, modeling its dynamics before default as  in the Black-Scholes  diffusion, and including a negative effect in the price each time a portion of the original position is sold in the market. 

We write the optimal execution problem as an stochastic optimal control problem  adopting the dynamic programming approach to solve it.  The HJB equation is deduced and a verification theorem is proved; furthermore, an explicit solution is found under specific assumptions in the model parameters. This approach have also been followed to analyze the problem of optimal dividend under default contagion by Qiu et al. \cite{QJL2023}.

The paper is organized as follows. In  Section \ref{section_preliminaries}, we  precisely describe the main elements of the model, including the default time and its intensity.  We also describe the price dynamics when affected by both the uncertainty inherent in the default time and the existence of a large supply of assets. In Section 3 we present the optimal control problem related with the optimal execution problem, motivating the variational form of the HJB equation, and providing the explicit for of its solution, together with an explanation of the resulting optimal execution policy. this is done  making some simplification in the function involved in the default time. This specific form of the default time is pedagogical and represents an interesting case study, as it allows us to infer the correct partitioning of the state space, through which the optimal strategy is defined. Numerical results illustrate these findings in Section \ref{sec:Numerical results}. The Appendix contains some technical results, which are presented at the end of the paper for ease of reading.

\section{Market impact model}\label{section_preliminaries}
In this section we describe the optimal execution model. Let us fix a filtered probability space $(\Omega,\mathcal{F},\mathcal{F}_t,\mathbb{P})$ satisfying the usual conditions and carrying a standard two dimensional $\Fb=(\mathcal{F}_t)$-Brownian motion $W=(W^1, W^2)$ and a random variable $U$, independent of $W$. %and an independent  Poisson process $N^{\gamma}$. 
Given $\rho\in[-1,1]$ fixed, define the correlated Brownian motions  $B_{t}=\rho W^1_t+\sqrt{1-\rho^2}W^2_t$, that is,  $[B,W^1]_{t}=\rho t$; the filtration generated by $B$ and $W$ is denoted by $\Fb^B\eqdef\{\mathcal{F}^B_t\}$ and $\Fb^{W^1,W^2}\eqdef\{\mathcal{F}_t^{W^1_t, W^2_t}\}$, respectively.

We consider an agent holding an initial (positive) position of  $y$ shares of a financial asset, which has to be sold maximising the expected gains. The information available to the agent is enclosed in the filtration $\{\mathcal{F}_t\}$. 
As part of our model, we include a random time $\tau$, which represents a default time, when we are concerned with the sell of stocks, or the disclosure of positive information of the firm, having a positive effect in the prices of the stock, if  we were interested in buying.  We shall consider the model adapted by Hull and White, motivated by the results of  Ettinger et al. \cite{EEH} 

\noindent
{\bf Default time.} We now present two different ways to model the default time, based on the assumption of the existence of an index that describes the financial state of the company, which, when crossed over a threshold, increases the possibility of bankruptcy. 
The first of these will be the case studied in this work, since its meaning in terms of a critical (deterministic) barrier of the index allows us to model interesting cases, in addition to allowing us to present the main results in a concrete way, as well as offering a geometric interpretation of the strategies in terms of a partition of the state space.

\begin{enumerate}
	\item[(a)] The first default $\tau$ is defined, in its general form,  as
	\begin{equation}\label{tau_1}
		\tau\eqdef\inf \left\{t>0: \lambda \int_0^t \psi\left(W^1_s-b(s)\right) \diff s>U\right\},
	\end{equation}
	where $\lambda>0$ is a constant rate,  $U$ is a random variable  independent  of $\mathcal{F}^{(W_1,W_2)}_\infty$, with  exponential law  of parameter one, and $b:\R_+\to\R_+$ is a continuous function. The function $\psi\in[0,1]$  is  right continuous and decreasing that satisfies $\lim_{x\rightarrow-\infty}\psi(x)=1$ and   $\lim_{x\rightarrow+\infty}\psi(x)=0$.
	However, we make the specific choice    $\psi(x)\eqdef \uno_{(-\infty,0)}(x)$,  considering   $$\psi(W^1_{t}-b)=\uno_{\{W^1_{t}<b\}}. $$ donde $b>0$ is a constant; see Ettinger et. al. \cite{EEH}.  Here $W^1_t$ represents the evolution of the credit index value of the firm. Note that when $\{W^1_{t}<b\}$ and $\{\tau> t\}$ describe the time periods prior to the company's default. 
	\item [(b)] The second option is to take the default time as $\tilde{\tau}\eqdef\inf\{t>0:W^1_{t}<b\}$, where $b>0$. This is somehow more drastic that the previous one, as it is expressed in terms of a barrier of the index process, unlike the previous one, in which there is an accumulation of time that the index process spends below the threshold value. In this case company's default is declared as soon as the index process falls below the barrier $b$.
\end{enumerate}

Define the {\it default process} $D_t=\uno_{[\tau> t]}$ 
%\uno_{[\tau\leq t]}$\blue{$\uno_{[\tau\geq t]}$? Mas a\'un podemos cambiarlo por la versi\'on cadlag es decir $\uno_{[\tau> t]}$ es decir en $t=\tau$ entramos en default} 
and its $\sigma-$field $\mathcal{D}_t=\sigma(D_s: s\leq t)$. 
Throghout it is assumed that the information available to the investor is contained in the completed version of the filtration  
$ \Gb=\Fb^{W^1,W^2}\vee\mathcal{D}$.
In this case, the process  
\begin{equation*}%\label{mart_Yor}
	M_t\eqdef D_t-\lambda \int_0^{\tau\wedge t}   \psi\left(W^1_s-b\right)\diff s
\end{equation*}
is a $\Gb-$ martingale, which implies that $t\to \lambda_t\eqdef\lambda  \psi\left(W^1_t-b\right) \uno_{[t< \tau]}$ is the $\Gb-$ intensity of $\tau$, and hence  $\langle D\rangle_t=\int_0^t \lambda_s \diff s$. The results presented in this paper consider only Case (a), with $\psi(W^1_{t}-b)=\uno_{\{W^1_{t}<b\}}$. 

Note that, defining 
\begin{equation}\label{def:lambda}
	h(r)\eqdef\lambda \int_0^r \psi\left(W^1_s-b\right) \diff s,
\end{equation}
for $t\geq r$, 
$$
\mathbb{P}[\tau>r| \mathcal{F}^{W_1,W_2}_t]=\exp\Bigg[- \lambda \int_0^r \psi\left(W^1_s-b\right)\diff s\Bigg]=\expo^{-h(r)}.
$$
%with $h(r)\eqdef\lambda \int_0^r \psi\left(W^1_s-b\right) \diff s$.
This follows from the following observation. Since $U$ is an  exponential random variable with parameter one, which is independent of
$\mathcal{F}^{W^1,W^2}_{\infty},$   it follows that
\begin{align*}
	\mathbb{P}[\tau>r| \mathcal{F}^{W_1,W_2}_t]&=\E[\uno_{\{\tau>r\}}|\mathcal{F}^{W_1,W_2}_t]\\
	&=\E[\uno_{\{h(r)\leq U\}}|\mathcal{F}^{W_1,W_2}_t]\\
	&=\E[\E[\uno_{\{h(r)\leq U\}}|F^{W_1,W_2}_{\infty}]|\mathcal{F}^{W_1,W_2}_t]\\
	&=\E[ \expo^{-h(r)}| \mathcal{F}^{W_1,W_{2}}_t]\\
	&=  \expo^{-h(r)}.
\end{align*}

%In fact, we have that  
%$$
%\mathcal{P}[\tau>r| \mathcal{F}^{W_1,W_2}_r]=\mathcal{P}[\tau>r| \mathcal{F}^{W_1,W_2}_t],
%$$
Hence, taking the limit when $t\to\infty$, we get 
\begin{align}\label{eq:estdefa}
	\mathbb{P}[\tau>r| \mathcal{F}^{W_1,W_2}_r]&=\mathcal{P}[\tau>r| \mathcal{F}^{W_1,W_2}_\infty]\notag\\
	&=\exp\Bigg[- \lambda \int_0^r \psi\left(W^1_s-b\right)\diff s\Bigg].
\end{align}

%\blue{[Aqui me surge una duda, no se si esta condici\'on se pueda extender a lo siguiente:
	%\begin{equation}\label{eq1}
	%\mathcal{P}[\tau>r| \mathcal{F}^B_{\infty}]=\exp\{- \lambda \int_0^r \psi\left(B_s-b(s)ds\right)\}
	%\end{equation}
	%\red{Considerando que $\{\tau>t\}=\{h(t)\leq U\}$ y el hecho de que $U$ es una exponencial independiente de $\mathcal{F}^{B}_{\infty}$ se sigue \eqref{eq1}, verdad?}
	%Quiz\'as bajo ciertas hip\'otesis se pueda tomando el l\'imite cuando $t\to\infty$, es decir
	%\begin{align*}
	%\mathcal{P}[\tau>r| \mathcal{F}^B_{\infty}]=\lim_{t\to\infty}\mathcal{P}[\tau>r| \mathcal{F}^B_t]&=\lim_{t\to\infty}\exp\{- \lambda \int_0^r \psi\left(B_s-b(s)ds\right)\}\notag\\
	%&=\exp\{- \lambda \int_0^r \psi\left(B_s-b(s)ds\right)\}.
	%\end{align*}
	%Esto implicar\'ia que
	%\begin{align}\label{cond_H}
	%\mathcal{P}[\tau>r| \mathcal{F}^B_{\infty}]=\mathcal{P}[\tau>r| \mathcal{F}^B_{t}].
	%\end{align}
	
	According to Bremaud and Yor \cite{BY}, this implies that the {\it H-}hypothesis holds, which establishes that   any $\mathbb{F}^{W_1,W_2}$ martingale is also a $\Gb-$martingale.
	Moreover, by the Jeulin-Yor theorem, the intensity process of $\tau$ is given by 
	\begin{equation}\label{def:intprocess}
		\lambda_t\eqdef\lambda  \psi\left(W^1_t-b\right) \uno_{[t< \tau]}.
	\end{equation}
	
	\noindent
	{\bf Investor's position.} 
	Given the initial amount of shares $y$ held by the investor, the trading strategies are denoted by  $\xi^{\sell}$, representing the amount of shares that have been sold  by the investor until time $t$. Hence, the total amount of shares held by the investor at time $t$ is given by
	\[
	Y_t\eqdef y-\xi^{\sell}_t, \qquad t\geq0,
	\]
	where   $\xi^{\sell}$ is a non-decreasing c\`agl\`ad  $(\mathcal{G}_t)$-adapted process such that
	\begin{align}\label{eq:hypstra}
		\xi^{\sell}_{0+}=0,%\qquad\text{and}  \qquad \lim_{t\to\infty}Y_t=0.
	\end{align}
	and $\xi^{\sell}_{\tau}=\xi^{\sell}_{t}$, for  $\;t>\tau$.
	We denote the set of strategies that satisfy the above assumptions by $\A(y)$. The description of the execution model is based on the work by Guo and Zervos \cite{ZG15}. 
	
	\begin{remark} (i) Strategies that combine the possibility of buying assets of stock before default can also be included in the model without much complication. However, as in the case studied by Guo and Zervos \cite{ZG15}, we expect that the optimal execution problem is reduce to the case when only selling is allowed.\\
		(ii) The hypothesis on the far right of (\ref{eq:hypstra}) establishes the fact that the investor aims to liquidate the complete position.   Further, after default $\tau$ occurs, no more transactions are allowed. 
		
	\end{remark}
	
	\noindent
	{\bf Stock price process.}
	The stock price observed by the agent, independently of the actions of other market participants, is modeled by the geometric Brownian motion $X^0$ with drift,
	which is known to satisfy the SDE:
	\[
	\frac{\diff X_t^0}{X_{t}^0} = \mu  \diff t + \sigma \diff B_t, \qquad t \geq 0,
	\]
	for some constants $\mu\in \R$ and $\sigma\neq 0$.
	This price process will be impacted in two different ways along the evolution of the market and the investor's execution strategy. Firstly, due to the oversupply of shares in the market due to the strategy $\xi^{\sell}\in\A(y)$ and, secondly, due to the company's default.
	Regarding the first one, 
	when the agent decides to sell  some  number of shares of the asset at time $t$, we assume that there is a multiplicative impact in the price, namely, the resulting price $\tilde{X}_t$ is assumed to have the form
	\begin{equation}\label{priceP}
		\tilde{X}_t=X^0_{t}\exp\{-\gamma \xi^{\sell}_t \},
	\end{equation}
	for some positive constant $\gamma$ describing the permanent impact on the price, in terms of the process $X^0_t$.  This is a c\`agl\`ad  process and we denote 
	$\Delta \xi^{\sell}_t=\xi^{\sell}_{t^+}- \xi^{\sell}_t$.
	Denoting the process $(\xi^{\sell})^{\cont}$    as the continuous part of $\xi^{\sell}$ 
	we get the dynamics of the price (affected by the transactions) as
	\begin{align}\label{eq4}
		\frac{\diff\tilde{X}_t}{\tilde{X}_{t}}=\left[\mu \diff t+\sigma \diff B_t \right]-\left[\diff (\xi^{\sell})_t^{\cont}+\int_0^{\Delta\xi_{t}^{\sell}}\expo^{-\gamma u}\diff u\right],\;\;t\geq 0.
	\end{align}
	Note that,   defining 
	\begin{align}\label{operator}
		\tilde{X}_{t}\circ_{\sell}\diff\xi_{t}^{\sell}&=\tilde{X}_{t}\diff (\xi^{\sell})_t ^{\cont}+\dfrac{1}{\gamma} \tilde{X}_{t}\Big[1-\expo^{-\gamma\Delta\xi_{t}^{\sell}}\Big]\notag\\
		&=\tilde{X}_{t}\diff (\xi^{\sell})_t^{\cont}+\tilde{X}_{t}\int_0^{\Delta\xi_{t}^{\sell}}\expo^{-\gamma u}\diff u,
	\end{align}
	the last term on the right of the above display can be written in a more compact way.

	Concerning the effect of the default in the price process, we assume that after the default time $\tau$ the price of the asset is zero, that is,
	\begin{equation}\label{GBM1}
		X_t \eqdef \tilde{X}_{t} \uno_{\{t \leq \tau\}}
	\end{equation}
	is the price process for a given investment strategy $\xi^{\sell}$.
	Therefore, the investor loses the remainder of their position once the default occurs, as $X_t = 0$ for $t > \tau$. Observe that $X^0_0=\tilde{X}_0=X_0=x\in \R^+$.
	
	\section{Stochastic control problem with full information}\label{sec_prob_formulation}
	
	Denoting by  $Y_t$  the number of shares left in the portfolio of the investor at time $t$, the pair $(\tilde{X}_t, Y_t)$ is referred as the state (controlled)  process associated to the strategy 
	$\xi^{\sell}$. 
	Let $C^{\sell}$ be a positive constant representing the transaction cost associated with the sell  of shares and $\delta>0$ the time discounting  factor. Then, denoting by $\E_{x,y,w}$  the conditional expectation on the event $\{\tilde{X}_0=x,Y_0=y,W^1_0=w\}$, the expected net present value of gains associated with each strategy $\xi^{\sell}\in \A(y)$ is given by
	\begin{align}\label{pc1.0}
		J_{x,y,w}(\xi^{\sell})&\eqdef\E_{x,y,w}\Biggr[\int_{[0,\tau)}\expo^{-\delta t}\left[\tilde{X}_{t}\circ_{\sell}\diff\xi^{\sell}_t-C^{\sell}\diff\xi_{t}^{\sell}\right]-\expo^{-\delta\tau}g(\tilde{X}_{\tau},Y_{\tau})\uno_{\{\tau<\infty\}}\Biggl].
	\end{align}
	The aim of the investor is to maximize the functional $J_{x,y,w}(\cdot)$ over the set $ \A(y)$. 
	
	Throughout   $g:\R^2_+\mapsto \R_+$ is a non-negative continuous function representing the terminal cost at the default time $\tau$, satisfying the boundary condition  $g(x,0)=0$,
	and such that, for $y\in B\subset \R^+$, with $B$ a compact set, there exists a positive constant, possibly depending on $B$, $C(B)$ such that $x\to g(x,y)\leq C(B) x$, for $x\geq 0$. Additionally, we also assume that
	\begin{align}\label{cond_drift}
		\delta>\mu.
	\end{align}

	%Note that $\tilde{X}_{0}=X^{0}_0=x>0$.
	Now, conditioning in the r.h.s. of (\ref{pc1.0}) by $\mathcal{F}^{W_1,W_2}_{\infty}$ and recalling the definition of $\tau$ as in \eqref{tau_1}, we obtain that
	\begin{align}\label{aux}
		J_{x,y,w}(\xi^{\sell})&=%\E_{x,y,w}\Biggr[\int_{[0,\tau)}\expo^{-\delta t}\left(\tilde{X}_{t }\circ_{\sell}\diff\xi^{\sell}_t-C_{\sell}\diff\xi_{t}^{\sell}\right)
		%-\expo^{-\delta\tau}g(\tilde{X}_{\tau-},Y_{\tau-  })\uno_{\{\tau<\infty\}}\Biggl]\notag\\
		%&=
		\E_{x,y,w}\Bigg[\int_{0}^{\infty}\expo^{-\lambda\int_0^t\psi(W^1_s-b)\diff s-\delta t}\left[\tilde{X}_{t}\circ_{\sell}\diff\xi^{\sell}_t-C^{\sell}\diff\xi_{t}^{\sell}\right]\notag\\&\;\;\; -\int_0^\infty \expo^{-\lambda\int_0^t\psi(W^1_s-b)\diff s-\delta t}\lambda\psi(W^1_{t}-b)g(\tilde{X}_{t },Y_{t })\diff t\Bigg]\notag\\
		&= \E_{x,y,w}\Bigg[\int_{0}^{\infty}\expo^{-\tilde{h}(t)}\Bigg[\tilde{X}_t\circ_{\sell}\diff\xi^{\sell}_t-C^{\sell}\diff\xi_{t}^{\sell} -\lambda\psi(W^1_{t}-b)g(\tilde{X}_{t},Y_{t})\diff t\bigg]\Bigg],
	\end{align}
	with the dynamics of $\tilde{X}$ described by (\ref{priceP}). Here $\tilde{h}(t)\eqdef h(t)+\delta t$; see (\ref{def:lambda}).
	
	We are interested in characterizing  the value function  defined by
	\begin{align*}%\label{eq13}
		V(x,y,w)=\sup_{\xi^s\in \A(y)}J_{x,y,w}(\xi^{\sell}).
	\end{align*}
	Notice that   $V(x,0,w)=0$ for any $x\geq 0$ and $w\in\R$.  Moreover, at the initial time the investor can sell immediately the available stock $y$, implementing the strategy with $\xi^s_0=y$,  that naturally induce the lower bound for the value function 
	\begin{equation}\label{eq:lowbound}
		V(x,y,w)\geq \frac{x}{\gamma}[1-\expo^{-\gamma y}] -C^{\sell} y.
	\end{equation}
	An upper estimate can also be obtained using  Ito's formula for
	\begin{align*}
		\diff(\expo^{-\tilde{h}(t) }  \tilde{X}_{t})&= \diff(\expo^{-\tilde{h}(t) } X^0_t \expo^{-\gamma \xi^s_t})\\
		=& \expo^{-\tilde{h}(t)}\left[ -(\lambda \psi(W^1_t-b)+\delta)X_t^0 \expo^{-\gamma \xi_t^s} \diff t +X^0_t \diff\expo^{-\gamma \xi^s_t} +
		\expo^{-\gamma \xi^s_t}(\mu X^0_t\diff t+\sigma X^0_t dB_t)\right]\\
		=&\expo^{-\tilde{h}(t)}\left[ (\mu-\delta-\lambda\psi(W^1_t-b))X^0_t \expo^{-\gamma \xi^s_t}\diff t+X^0_t \diff \expo^{-\gamma \xi^s_t}+\sigma X^0_t \expo^{-\gamma \xi^s_t} \diff B_t  \right].
	\end{align*}
	Observe that extracting the second term on the r.h.s., we get for fixed finite  $T$,
	\begin{align*}
		\int_0^T \expo^{-\tilde{h}(t)}  \tilde{X}_t\circ_{\sell}\diff\xi^{\sell}_t &= -\frac{1}{\gamma}\int_0^T \expo^{-\tilde{h}(t)}  X^0_t \diff \expo^{-\gamma \xi^s_t}=\frac{x}{\gamma}-\frac{1}{\gamma}\expo^{-\tilde{h}(T)}  \tilde{X}_{T}\\
		&\quad+\frac{1}{\gamma}\int_0^T \expo^{-\tilde{h}(t) } X^0_t \expo^{-\gamma \xi^s_t} \left[ (\mu-\delta-\lambda\psi(W^1_t-b)) \diff t +\sigma \diff B_t \right].
	\end{align*}
	In view of the hypotheses on the parameters (\ref{cond_drift}), we get from (\ref{aux}) that
	$J_{x,y,w}(\xi^{\sell})\leq \frac{x}{\gamma},$
	and, since $\xi^{\sell}$  was chosen arbitrarily, we conclude that 
	$$V(x,y,w) \leq \frac{x}{\gamma}.$$
	The above estimations of the value function $V$ will be useful to characterize it through the HJB equation, once the verification result has been established.
	
	We define the operator $\mathcal{L}$ acting on functions $f\in\mathcal{C}^{2,1,2}(\R\times \R_+\times\R)$ by
	\begin{align*}
		\mathcal{L}f(x,y,w)%\eqdef \frac{1}{2}\sigma^2x^2f_{xx}(x,y)+\mu f_x(x,y)-\left(\lambda \psi(b-b_0)+\delta\right)f(x,y),\\
		\eqdef& \frac{1}{2}\sigma^2x^2 f_{xx}(x,y,w)+\mu x f_x(x,y,w)+\frac{1}{2}f_{ww}+\rho\sigma x f_{xw}-\delta f(x,y,w).
	\end{align*}

	In order to derive heuristically  the HJB equation, we follow the next argument adapted to our problem  from \cite[p. 295]{ZG15}. Assume that at time $t>0$ the stock price is $x$, the number of stocks left in the portfolio is $y$ and the index value is $w$. Then,  using (\ref{eq:estdefa}),  there are two possible scenarios, described by the possibility that, for $\Delta t>0$, 
	\begin{enumerate}
		\item $t<\tau\leq t+\Delta t$, which occurs with an approximate (conditional) probability given by $\Delta t \lambda_t$.
		
		\item $\tau> t+\Delta t$,  which occurs with an approximate (conditional) probability given by $1-\Delta t \lambda_t$.
	\end{enumerate}
	Then, a possible investment strategy is that the investor   sells all the available  stock in the first scenario and, in the second, do not sell anything. In this case, 
	applying this argument at $t=0$, given the homogenity of the model,
	and based in the expression (\ref{pc1.0}) for $J_{x,y,w}$, we get 
	\begin{multline}\label{Est1}
		v(x,y,w)\geq  \E_{x,y,w}\left[\expo^{-\int_0^{\Delta t} (\lambda_s+\delta)\diff s}v(X^{0}_{\Delta t},y,W^1_{\Delta t})\uno_{[\tau> \Delta t]}\right.\\
		\left.+ \expo^{-\int_0^{\Delta t} \lambda_s \diff s} g(X^{0}_{\Delta t},y)\uno_{[\tau> \Delta t]}+\left[ \frac{x}{\gamma}[1-\expo^{-\gamma y}] -C^{\sell} y\right]\uno_{\{\tau\leq \Delta t\}} \right].
		%-\E_{x,y,w}\left[  g(\red{X^{0}_{\tau}},y)\uno_{\{\tau\leq \Delta t\}}\right].
	\end{multline} 
	Here, we have used the hypothesis on $g$ and  (\ref{eq:lowbound}). 
	Then, using Ito's rule and letting $\Delta t$ go to zero, we get the inequality
	$$
	\mathcal{L}v(x,y,w)-\lambda\psi(w-b)v(x,y,w)-\lambda\psi(w-b) g(x,y)\leq 0.
	$$

	The previous strategy can be modified slightly,  allowing to sell a small fraction of stock in the second scenario and, again, sell all the stock in the first scenario.  Then, the last term in r.h.s. of (\ref{Est1}) remains the same, and hence
	$$
	v(x,y,w)\geq v(x-\varepsilon\gamma x,y-\varepsilon,w)+(x-C^{\sell})\varepsilon.
	$$
	Then, by taking $\varepsilon\downarrow0$ we obtain
	\[
	-\gamma xv_x(x,y,w)-v_y(x,y,w)+x-C^{\sell}\leq 0.
	\]

	Then, the dynamic programming equation is given by
	\begin{multline}\label{HJB1}
		\max\{\mathcal{L}v(x,y,w)-\lambda\psi(w-b)[v(x,y,w)+g(x,y)],\\ -\gamma xv_x(x,y,w)-v_y(x,y,w)+x-C^{\sell}\}=0.
	\end{multline} 
	%\blue{[Creo que deber\'ia ser
		%	\begin{align}\label{HJB2}
			%		\max\{\mathcal{L}v(x,y,w)-\lambda\psi(w-b)v+\lambda \psi(w-b)g(x,y), -\gamma xv_x(x,y,w)-v_y(x,y,w)+x- C^{\sell}\}=0.
			%	\end{align} 
		%	]}
	with the boundary condition
	\begin{align}\label{eq:Boundcon}
		\begin{cases}
			v(x,0,w)=0,\qquad\text{for every $x>0$},\;w\in\R\\
			v(0,y,w)=0,\qquad\text{for every $y>0$},\;w\in\R.\ %\red{\text{[Es necesaria esta hip\'otesis?]}}.
		\end{cases}
	\end{align}

	\subsection{Explicit solution to the HJB equation}	
	In the next result we shall prove that the solution of the HJB equation (\ref{HJB1}) can be obtained explicitly under suitable conditions by constructing  a smooth solution when the  liquidation cost is bilinear; namely, $g(x,y)=Kxy$ for some  positive constant $K$.  This is achieved by constructing a candidate solution and identifying the continuation and selling regions in $\R_+\times\R_+\times\R$, denoted by $\mathbb{W}$ and $\mathbb{S}$, respectively. 
	
	Roughly speaking, the continuation region $\mathbb{W}$ is characterized as the interior of the set where
	\begin{equation}\label{HJB_0_wait}
		\mathcal{L}v(x,y,w)-\lambda\psi(w-b)[v(x,y,w)+g(x,y)]=0,
	\end{equation}
	while the selling region $\mathbb{S}$ is the interior of the set where the gradient constraint
	\begin{equation}\label{HJB_0_sell}
		-\gamma xv_x(x,y,w)-v_y(x,y,w)+x-C^{\sell}=0 
	\end{equation}
	is binding. The threshold value $b$, which appears in the definition of the intensity process $\{\lambda_t\}$  (see (\ref{def:intprocess})) associated with the default time
	$\tau$ will play an important role in the description of these regions, and hence, the investor's optimal actions shall depend on the position of the  index value process $\{W^1_t\}$ with respect to  this point, as it is explained below. The waiting and selling regions ($\mathbb{W}$ and $\mathbb{S}$, resp.) are described next in terms of the functions $F, G_\lambda:\R^+\times \R\to \R^+$ defined in Appendix \ref{proof_HJB_app}, representing the reduction in the price once the transaction occurs:
	\begin{align*}%\label{S1}
		\begin{split}
			\mathcal{W}:&=\{(x,y,w)\in\R_+\times\R_+\times\R:y>0,\ x<F(y,w), \ w{\geq}b \},\\
			\mathcal{S}:&=\{(x,y,w)\in \R_+\times\R_+\times\R:y>0,\ x\geq F(y,w), \ w{\geq}b\},\\
			\mathcal{W}^{(\lambda)}:&=\{(x,y,w)\in \R_+\times\R_+\times\R:y>0,\ x<G_{\lambda}(y,w), \ w{<}b\},\\
			\mathcal{S}^{(\lambda)}:&=\{(x,y,w)\in \R_+\times\R_+\times\R:y>0,\ x\geq G_{\lambda}(y,w), \ w{<}b\}.
		\end{split}
	\end{align*}
	The regions $\mathcal W$ and $\mathcal W^{(\lambda)}$ represent the waiting regions
	in the absence and (possible) presence of default, respectively, while $\mathcal S$ and
	$\mathcal S^{(\lambda)}$ denote the corresponding selling regions under those scenarios; we implicitly assume that (\ref{HJB_0_wait}) and (\ref{HJB_0_sell}) are satisfied in the corresponding region. Define the function
	\begin{equation}\label{fun_Y}
		\mathbb{Y}(x,z)\eqdef\frac{1}{\gamma}\ln\frac{x}{z},\qquad x,z>0.
	\end{equation}
	
	We shall see that in the selling
	regions, the agent liquidates shares according to the mappings
	$\mathbb{Y}(x,F(y,w))$ and $\mathbb{Y}(x,G_\lambda(y,w))$, which account for the possibility
	of default. In fact, each selling region can be decomposed into two disjoint subregions,
	$\mathcal S=\mathcal S_1\cup\mathcal S_2$ and
	$\mathcal S^{(\lambda)}=\mathcal S^{(\lambda)}_1\cup\mathcal S^{(\lambda)}_2$ given by:
	\begin{align*}
		\mathcal{S}_{1}&\eqdef\{(x,y,w)\in\R_{+}\times\R_{+}\times\R:y>0,\ x>F_{0}\expo^{\gamma y},\ w\geq b\},\notag\\
		\mathcal{S}_{2}&\eqdef\{(x,y,w)\in\R_{+}\times\R_{+}\times\R:F_{0}\leq x\leq F_{0}\expo^{\gamma y},\ y>0,\  w\geq b\},\notag\\
		\mathcal{S}^{(\lambda)}_{1}&\eqdef\{(x,y,w)\in\R_{+}\times\R_{+}\times\R:y>0,x>\expo^{\gamma y}G_{\lambda}(y),w<b\},\notag\\
		\mathcal{S}^{(\lambda)}_{2}&\eqdef\{(x,y,w)\in\R_{+}\times\R_{+}\times\R:y>0,\ G_{\lambda}(y)\leq x\leq \expo^{\gamma y}G_{\lambda}(y),\ w<b\},
	\end{align*}
	where $F_{0}>0$ and $G_{\lambda}(y)$  shall be defined later on; see Theorem \ref{HJB_proof}. In $\mathcal S_1$ and $\mathcal S^{(\lambda)}_1$, the agent liquidates all remaining
	shares immediately (see Figure \ref{f1_0}), whereas in $\mathcal S_2$ and $\mathcal S^{(\lambda)}_2$ the
	agent performs a partial liquidation so that the stock price is reduced to
	$F_{0}$ or $G_\lambda(y)$, respectively; for more details on this structure of the regions see Appendix \ref{proof_HJB_app}.
	Intuitively, the presence of default risk induces more aggressive liquidation
	strategies. Consequently, we expect
	$\mathcal W^{(\lambda)}\subset\mathcal W$ and
	$\mathcal S_1\subset\mathcal S^{(\lambda)}_1$, meaning that the waiting region
	shrinks while the region corresponding to immediate liquidation expands when
	default is possible. 
	
	\begin{figure}[h!]
		\centering
		\includegraphics[scale=0.7]{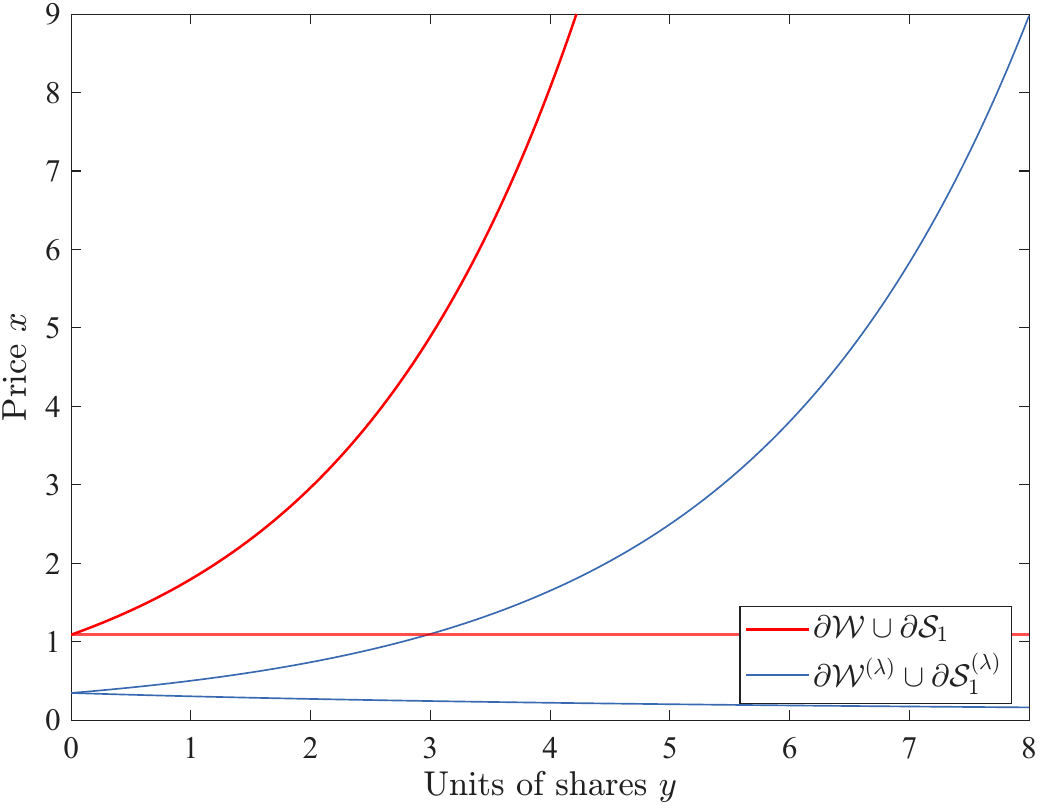}
		\caption{}\label{f1_0}
	\end{figure}
	We state now one of the main results of this paper, which proof is deferred to Appendix~\ref{proof_HJB_app}.
	\begin{theorem} \label{HJB_proof}
		The HJB equation defined in \eqref{HJB1} with boundary conditions as in \eqref{eq:Boundcon} has a solution $v$ that belongs to $\mathcal{C}^{2,1,2}\big(\R_+\times\R_+\times(\R\setminus\{b\})\big)$, which is characterized by
		\begin{equation}\label{eq12}
			v(x,y,w)=v^{(0)}(x,y)\uno_{\{w\geq b\}}+v^{(\lambda)}(x,y)\uno_{\{w<b\}},
		\end{equation}
		%\noident	 
		where 
		\begin{align}
			v^{(0)}(x,y)&=
			\begin{cases}
				\frac{x^{n_0}}{\gamma n_0^2}\left[\frac{n_0-1}{n_0 C^{\sell}}\right]^{n_0-1}[1-\expo^{-\gamma n_0y}]&\text{if}\ 0<x\leq F_{0},\\
				\vspace{-0.4cm}&\\
				v^{(0)}(F_0,y-\mathbb{Y}(x,F_0))&\\
				\quad+\frac{x}{\gamma}\left[1-\expo^{-\gamma \mathbb{Y}(x,F_0)}\right]- C^{\sell}\mathbb{Y}(x,F_0)&\text{if}\ F_{0}<x\leq F_{0}\expo^{\gamma y},\\
				\vspace{-0.4cm}&\\
				\frac{x}{\gamma}\left[1-\expo^{-\gamma y}\right]- C^{\sell}y&\text{if}\ x>F_{0}\expo^{\gamma y},\\
			\end{cases}\label{sol1}\\
			v^{(\lambda)}(x,y)&=
			\begin{cases}
				x^{n_1}B(y)-\frac{\lambda K}{\delta+\lambda-\mu}xy&\text{if}\ 0<x\leq G_{\lambda}(y),\\
				\vspace{-0.4cm}&\\
				v^{(\lambda)}(G_{\lambda}(y),y-\mathbb{Y}(x,G_{\lambda}(y)))&\\
				\quad+\frac{x}{\gamma}\left[1-\expo^{-\gamma \mathbb{Y}(x,G_{\lambda}(y))}\right]- C^{\sell}\mathbb{Y}(x,G_{\lambda}(y))&\text{if}\ G_{\lambda}(y)< x\leq G_{\lambda}(y)\expo^{\gamma y},\\
				\vspace{-0.4cm}&\\
				\frac{x}{\gamma}\left[1-\expo^{-\gamma y}\right]- C^{\sell}y&\text{if}\ x>G_{\lambda}(y)\expo^{\gamma y},
			\end{cases}\label{sol2}
		\end{align}
		with $n_0>1$ and $n_{1}>1$   solutions to the equations, respectively,
		\begin{align}
			\frac{1}{2}\sigma^2 l[l-1]+\mu l-\delta=0,\notag\\
			\frac{1}{2}\sigma^2 l[l-1]+\mu l-[\lambda+\delta]=0,\label{eq7}
		\end{align}
		and
		\begin{align}
			F_0&\eqdef\frac{n_0 C^{\sell}}{n_0-1},\label{eq6}\\
			G_{\lambda}(y)&\eqdef\frac{n_1 C^{\sell}}{[n_1-1]\left[1+\frac{\lambda K}{\delta+\lambda-\mu}(\gamma y+1)\right]},\label{eq10}\\
			B(y)&\eqdef\expo^{-\gamma n_1y}\int_0^y\frac{ C^{\sell}\expo^{\gamma n_1u}}{[n_1-1][G_{\lambda}(u)]^{n_{1}}}\diff u.\label{eq10.1}
		\end{align}
	\end{theorem}
	\begin{remark} Note that the solution $v$ given in Theorem~\ref{HJB_proof} satisfies, for all $(x,y,w)\in\R_+\times\R_+\times\R$,
		\begin{align}\label{cond_bound} 
			\frac{x}{\gamma}[1-\expo^{-\gamma y}] -C^{\sell} y\leq v(x,y,w)\leq \frac{x}{\gamma}.
		\end{align}
	\end{remark}
	\begin{theorem}
		Consider the optimal execution problem formulated in Section \ref{sec_prob_formulation}. The function $v$ defined in Theorem~\ref{HJB_proof} satisfies that  
		\[
		v(x,y,w)\geq V(x,y,w), \qquad\text{for all $(x,y,w)\in\R_+\times\R_+\times\R$}.
		\] 
		Moreover, if for all initial conditions $(x,y,w)\in\R_+\times\R_+\times\R$, there exists $\xi^*\in \Xi(y)$ such that 
		\begin{align*}%\label{optimal_strategy}
			(\tilde{X}_t^*,Y_t^*,W_t)\in\mathcal{W},\qquad \text{for all $t\geq0$, $\mathbb{P}$-a.s.}
		\end{align*}
		and,
		\begin{align}\label{optimal_strategy_2}
			\xi^*_{t+}=\int_{[0,t]}1_{\{(\tilde{X}_t^*,Y_t^*,W^*_t)\in\mathcal{S}\}}\diff \xi^*_t,\qquad \text{for all $t\geq0$, $\mathbb{P}$-a.s.,}
		\end{align}
		where $\tilde{X}^*$, and $Y^*$ are the share price and shares held processes associated with the liquidation strategy $\xi^*$. Then $v$ identifies with the value function $V$ of the stochastic control problem, i.e.
		\[
		v(x,y,w)=V(x,y,w), \qquad (x,y,w)\in\R_+\times\R_+\times\R.
		\]

	\end{theorem}
	
	\begin{proof}
		(i) Consider the function $v$ defined in \eqref{eq12}, and an arbitrary strategy $(\xi^s)\in\A(y)$. 
		Note that $v\in\mathcal{C}^{2,1,2}\big(\R_+\times\R_+\times(\R\setminus\{b\})\big)$, and therefore $v$ is not continuous at $w=b$. However, since the Brownian motion $W$ does not spend positive time at $b$, $\mathbb{P}$-a.s., we may apply It\^o's formula to the process $\{\expo^{ -\tilde{h}(t)}v(\tilde{X}_t,Y_t,W^1_t)\}$, with $\tilde{h}(t)\eqdef h(t)+\delta t$, and obtain
		\begin{align}\label{HJB}
			\expo^{-\tilde{h}(t)}&v(\tilde{X}_{t^+},Y_{t^+},W^1_t)=v(x,y,w)\notag\\&+\int_0^t\expo^{-\tilde{h}(s)}\left(\mathcal{L}v(\tilde{X}_{s},Y_{s},W^1_s)-\lambda \psi(W^1_s-b)v(\tilde{X}_{s},Y_{s},W^1_s)\right)\diff s\notag\\
			& -\int_0^t\expo^{-\tilde{h}(s)}\left[\gamma \tilde{X}_{s} v_x(\tilde{X}_{s},Y_{s},W^1_s)+v_y(\tilde{X}_{s},Y_{s},W^1_s)\right]\diff(\xi^s)^{c}_{s}\notag\\
			&+\sum_{0\leq s\leq t}\expo^{-\tilde{h}(s)}\left[v(\tilde{X}_{s+},Y_{s+},W^1_s)-v(\tilde{X}_{s},Y_{s},W^1_s)\right] +\tilde{M}_t,
		\end{align}
		where 
		\begin{align*}
			\tilde{M}_t\eqdef\int_0^t\expo^{-\tilde{h}(s)}\Big[\tilde{X}_{s }v_x(\tilde{X}_{s },Y_{s },W^1_s)\sigma \diff B_s+v_w(\tilde{X}_{s },Y_{s },W^1_s)\diff W^1_s\Big]
		\end{align*}
		is a local martingale.
		Following the proof of Proposition 4.1 of Guo and Zervos\cite{ZG15} and using the fact that $v$ satisfies \eqref{HJB1}, we  get
		\begin{align}\label{jump}
			v&(\tilde{X}_{s^+},Y_{s^+},W^1_s)-v(\tilde{X}_{s},Y_{s},W^1_s)\notag\\
			&\quad=-\int_0^{\Delta \xi_s^s}\Bigg[\gamma \expo^{-\gamma u}\tilde{X}_{s}v_x(\expo^{-\gamma u}\tilde{X}_{s},Y_{s}-u,W^1_s)\notag\\
			&\quad\quad+v_y(\expo^{-\gamma u}\tilde{X}_{s},Y_{s}-u,W^1_s)\Bigg]\diff u\leq \int_0^{\Delta \xi_s^s} [C^{\sell}- \expo^{-\gamma u}\tilde{X}_s]\diff u.
		\end{align}
		Hence, using \eqref{jump} in \eqref{HJB}, we have that
		\begin{align*}
			&\expo^{-\tilde{h}(t)}v(\tilde{X}_{t^+},Y_{t^+},W^1_t)\leq v(x,y,w)\notag\\&+\int_0^t\expo^{-\tilde{h}(s)}\left(\mathcal{L}v(\tilde{X}_{s},Y_{s},W^1_s)-\lambda \psi(W^1_s-b)v(\tilde{X}_{s},Y_{s},W^1_s)\right)\diff s\notag\\
			&+\int_0^t\expo^{-\tilde{h}(s)} [C^{\sell}- \tilde{X}_s] \diff(\xi^s)^{c}_{s}  + \sum_{0\leq s\leq t}\expo^{-\tilde{h}(s)}  \int_0^{\Delta \xi_s^s} [C^{\sell}- \expo^{-\gamma u}\tilde{X}_s]du+\tilde{M}_{t}.
		\end{align*}
		%& -\int_0^t\expo^{-\tilde{h}(s)}\left[\gamma \tilde{X}_{s-}v_x(\tilde{X}_{s-},Y_{s-},W^1_s)+v_y(\tilde{X}_{s-},Y_{s-},W^1_s)\right]\diff(\xi^s)^{c}_{s} +\tilde{M}_{t}\notag\\
		%&-\sum_{0\leq s\leq t}\expo^{-\tilde{h}(s)}\int_0^{\Delta \xi_s^s}\left[\gamma \expo^{-\gamma u}\tilde{X}_sv_x(\expo^{-\gamma u}\tilde{X}_u,Y_u-u,W^1_u)+v_y(\expo^{-\gamma u}\tilde{X}_u,Y_u-u,W^1_u)\right]\diff u.
		Using  again that $v$ is a solution to \eqref{HJB1} and (\ref{operator}), we obtain that
		\begin{align*}
			v(x,y,w)+\tilde{M}_{t}&\geq \expo^{-\tilde{h}(t)}v(\tilde{X}_{t^+},Y_{t^+},W^1_t)\notag\\
			&\quad+\int_{0}^{t}\expo^{-\tilde{h}(s)}\left[\tilde{X}_s\circ_{\sell}\diff\xi^{\sell}_s-C^{\sell}\diff\xi_{s}^{\sell}-\lambda\psi(W^1_{s}-b)g(\tilde{X}_{s},Y_{s})\diff s\right].
		\end{align*}
		By using stopping time $t\wedge \tau_n$, with $\tau_n=\inf\{t\geq 0\;|\;(X_t,W^1_t)\notin B_n\}$, where  $B_n$ is the ball of radius $n$ in $\R^2$, and taking expectations
		
		%By \eqref{cond_bound} and \eqref{cond_drift} we have that $\tilde{M}$ is a martingale, so taking expectation in the previous inequality gives
		\begin{align*}%\label{ver_lemma_2}
			v(x,y,w)&\geq \E_{x,y,w}\left[\expo^{-\tilde{h}(t\wedge \tau_n)}v(\tilde{X}_{(t\wedge \tau_n)^+},Y_{(t\wedge\tau_n)^+},W^1_{t\wedge\tau_n})\right]\notag\\
			&\quad+\E_{x,y,w}\left[\int_{0}^{t\wedge\tau_n}\expo^{-\tilde{h}(s)}\left[\tilde{X}_s\circ_{\sell}\diff\xi^{\sell}_s-C^{\sell}\diff\xi_{s}^{\sell}-\lambda\psi(W^1_{s}-b)g(\tilde{X}_{s},Y_{s})\diff s\right]\right].
		\end{align*}
		Now we can apply Fatou's lemma, taking the liminf as $n\to\infty$, and using the hypotheses on $g$, obtaining 
		\begin{align}\label{ver_lemma_1}
			v(x,y,w)&\geq \E_{x,y,w}\left[\expo^{-\tilde{h}(t)}v(\tilde{X}_{t^+},Y_{t^+},W^1_{t})\right]\notag\\
			&\quad+\E_{x,y,w}\left[\int_{0}^{t}\expo^{-\tilde{h}(s)}\left[\tilde{X}_s\circ_{\sell}\diff\xi^{\sell}_s-C^{\sell}\diff\xi_{s}^{\sell}-\lambda\psi(W^1_{s}-b)g(\tilde{X}_{s},Y_{s})\diff s\right]\right].
		\end{align}
		%	Using the l.h.s. of (\ref{cond_bound}), the first term of the previous display is bounded below by 
		%	$$
		%	-C^{\sell} \E_{x,y,w} [Y_t],
		%	$$
		%	and taking the limit when $t \to \infty$, we get the result, using the fact that $Y_t$ is bounded and $\lim_{t\to \infty} Y_t =0$. \blue{[No entiendo bien para que esta condici\'on y adem\'as creo que no es necesario pedir que $\lim_{t\to \infty} Y_t =0$. Cuando se dice que se obtiene el resultado, c\'ual exactamente? No se puede omitir y mejor decir que por \eqref{cond_bound} 
			%	\begin{align}\label{ver_lemma_1}
				%		v&(x,y,w)\geq \E_{x,y,w}\left[\expo^{-\tilde{h}(t)}v(\tilde{X}_{t^+},Y_{t^+},W^1_{t})\right]\notag\\
				%		&+\E_{x,y,w}\left[\int_{0}^{t}\expo^{-\tilde{h}(s)}\left[\tilde{X}_s\circ_{\sell}\diff\xi^{\sell}_s-C^{\sell}\diff\xi_{s}^{\sell}-\lambda\psi(W^1_{s}-b)g(\tilde{X}_{s},Y_{s})\diff s\right]\right]\notag\\
				%		&\geq-C^{\sell} \E_{x,y,w}\left[\expo^{-\tilde{h}(t)}Y_t+\right]\notag\\
				%		&+\E_{x,y,w}\left[\int_{0}^{t}\expo^{-\tilde{h}(s)}\left[\tilde{X}_s\circ_{\sell}\diff\xi^{\sell}_s-C^{\sell}\diff\xi_{s}^{\sell}-\lambda\psi(W^1_{s}-b)g(\tilde{X}_{s},Y_{s})\diff s\right]\right],
				%	\end{align}
			%	Y como $Y$ es acotado y por \eqref{cond_drift}
			%	$$
			%	\lim_{t\to\infty}\E_{x,y,w}\left[\expo^{-\tilde{h}(t)}Y_t+\right]=0?
			%	$$ 	
			%	Y por ende se obtiene
			%	\begin{align*}
				%v(x,y,w)\geq\E_{x,y,w}\left[\int_{0}^{\infty}\expo^{-\tilde{h}(s)}\left[\tilde{X}_s%\circ_{\sell}\diff\xi^{\sell}_s-C^{\sell}\diff\xi_{s}^{\sell}-\lambda\psi(W^1_{s}-b)g(\tilde{X}_{s},Y_{s})\diff s\right]\right]?
				%\end{align*}
				%	 ]}
			%
			Finally, using \eqref{cond_bound} and \eqref{cond_drift}, we can use dominated convergence to conclude that 
			$$
			\lim_{t\to\infty}\E_{x,y,w}\left[\expo^{-\tilde{h}(t)}v(\tilde{X}_{t^+},Y_{t^+},W^1_t)\right]=0.
			$$ 
			Hence, taking $t\uparrow\infty$ in \eqref{ver_lemma_1}
			we obtain that $v(x,y,w)$ is bounded below by 
			\begin{align*}
				&\E_{x,y,w}\left[\int_{0}^{\infty}\expo^{-\tilde{h}(s)}\left[X_s\circ_{\sell}\diff\xi^{\sell}_s-C^{\sell}\diff\xi_{s}^{\sell}-\lambda\psi(W^1_{s}-b)g(X_{s},Y_{s})\diff s\right]\right]\notag\\
				&=\E_{x,y,w}\Biggr[\int_{[0,\tau)}\expo^{-\delta t}\left(X_t\circ_{\sell}\diff\xi^{\sell}_t-C^{\sell}\diff\xi_{t}^{\sell}\right)
				-\expo^{-\delta\tau}g(X_{\tau-},Y_{\tau-})\uno_{\{\tau<\infty\}}\Biggl].
			\end{align*}
			Thus, the fact that the strategy $\xi\in \mathcal{A}(y)$ is arbitrary, implies that
			\begin{align}\label{optimal_strategy_4}
				V(x,y,w)\leq v(x,y,w), \quad \text{for}\; (x,y,w)\in\R_+\times\R_+\times\R.
			\end{align}
			(ii) On the other hand, note that the strategy $\xi^*$ given by \eqref{optimal_strategy_2} has a performance criterion given by the function $v$ defined in Theorem~\ref{HJB_proof}, which is a solution to \eqref{HJB1}. 
			Then, following steps similar to those used in the derivation of the upper bound for the value function in \eqref{optimal_strategy_4}, we obtain that
			\begin{align}\label{optimal_strategy_6}
				&\expo^{-\tilde{h}(t)}v(\tilde{X}_{t^+}^*,Y_{t^+}^*,W^1_t)= v(x,y,w)+\int_0^t\expo^{-\tilde{h}(s)}\lambda\psi(W_s^1-b)g(\tilde{X}^*_s,Y_s^*)\diff s\notag\\
				&-\int_0^t\expo^{-\tilde{h}(s)}\left[\tilde{X}_{s}^*-C^{\sell}\right]\diff (\xi^s)^{c}_{s} -\sum_{0\leq s\leq t}\expo^{-\tilde{h}(s)}\int_0^{\Delta \xi_s^s}\left[\tilde{X}_u^*-C^{\sell}\right]\diff u+\tilde{M}_{t}^{*}. 
			\end{align}
			Hence, taking expectations in \eqref{optimal_strategy_6} and proceeding like in Step (i) we conclude that $v(x,y,w)$ is iqual to
			\begin{align}\label{optimal_strategy_5}
				&\E_{x,y,w}\left[\int_{0}^{\infty}\expo^{-\tilde{h}(s)}\left[\tilde{X}^*_s\circ_{\sell}\diff\xi^{*}_s-C^{\sell}\diff\xi_{s}^{*}-\lambda\psi(W^1_{s}-b)g(\tilde{X}^*_{s},Y_{s}^*)\diff s\right]\right]\notag\\
				&=\E_{x,y,w}\Biggr[\int_{[0,\tau)}\expo^{-\delta t}\left(\tilde{X}_t^*\circ_{\sell}\diff\xi^{*}_t-C^{\sell}\diff\xi_{t}^{*}\right)
				-\expo^{-\delta\tau}g(\tilde{X}^*_{\tau-},Y^*_{\tau-})\uno_{\{\tau<\infty\}}\Biggl]\notag\\
				&=J_{x,y,w}(\xi^{*})\leq V(x,y,w).
			\end{align}
			Therefore, by \eqref{optimal_strategy_4} and \eqref{optimal_strategy_5} we obtain
			\[
			V(x,y,w)=J_{x,y,w}(\xi^{*})=v(x,y,w),
			\]
			which shows the optimality of the strategy $\xi^*$, proving the claim.
		\end{proof}

		%\subsection{Analysing the limiting behaviour wrt to the default pa\-ra\-me\-ter $\lambda$}
		%\red{[A lo mejor hay que dar mas detalle de los l\'imites]}  %Since the variable $w$ affects $v$ only through the occurrence of default, and hence determines which expression, \eqref{sol1} or \eqref{sol2}, applies, we set $V(x,y)=v(x,y,w)$   for $w>b$, where $v$ is given by \eqref{sol1}, and  $V_{\lambda}(x,y)=v(x,y,w)$ for $w\leq b$, where  $v$ is given by \eqref{sol2}.   
		\begin{remark}
			Notice that, by the definitions of $n_1$, $G_{\lambda}$, and $B$ in \eqref{eq7}--\eqref{eq10.1}, as $\lambda\downarrow 0$ we have the following limits, for $y\in\R_+$,
			\[
			n_{1}\downarrow n_{0}, \qquad 
			G_{\lambda}(y)\to F_0, \qquad 
			B(y)\to \frac{C^s}{\gamma(n_0-1)F_0^{n_0}}\bigl[1-e^{-\gamma n_0 y}\bigr]x.
			\]
			This implies that
			\[
			\lim_{\lambda\downarrow 0} v(x,y,w)=v^{(0)}(x,y),
			\qquad \text{for } (x,y,w)\in\R_+\times\R_+\times\R,
			\]
			where $v^{(0)}$ is given by \eqref{sol1} which coincides, with equation (15) of Guo and Zervos\cite{ZG15}.
			Thus, we recover the results in their Proposition~5.1 
			for the case in which default is not considered in the optimal execution problem.
		\end{remark}

		\section{Numerical results}\label{sec:Numerical results}
		
		In this section, we present numerical examples illustrating  the results obtained in the proceeding sections. We consider here  that the drift and the volatility in \eqref{eq4} are given by $\mu=0.5$ and $\sigma=0.2$. Moreover, the discount rate is $\delta=0.7$,  and  the multiplicative-impact parameter is $\gamma=0.5$. The transaction cost for liquidating the  shares is $C^{\sell}=0.3$, and the cost associated  with default is  $K=0.5$.
		
		Figures \ref{f1} and \ref{f2} display the behaviour of the maps $x\mapsto v(x,y,w)$ (blue line), for the cases that $w\geq b$ and $w< b$, respectively,  for  fifty values  of $y$ evenly spaced on $[0.1,7]$ (with  default parameter $\lambda=0.7$ in Figure \ref{f2}). Recall, the functions $v(x,y,w)=v^{(0)}(x,y)$ (if $w\geq 0$) and $v(x,y,w)=v^{(\lambda)}(x,y)$ (if $w<b$), are defined in \eqref{sol1} and \eqref{sol2}. 
		
		\begin{figure}[h!]
			\centering
			\includegraphics[scale=0.7]{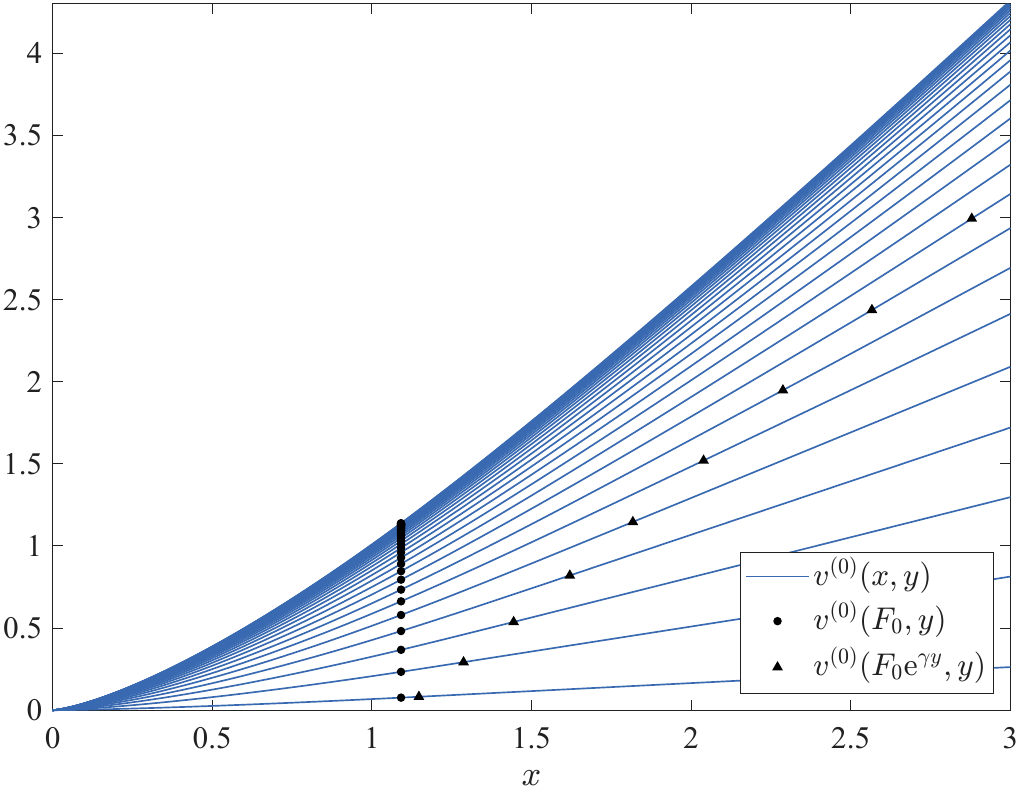}
			\caption{}\label{f1}
		\end{figure}
		In addition, the figures show pointwise evaluations as functions of $y$: In Figure \ref{f1} the black circular makers correspond to  $y\mapsto v^{(0)}(F_{0},y)$, where $F_{0}\approx1.0914$ is determined from \eqref{eq6},  and the black triangular makers to $y\mapsto v^{(0)}(F_{0}\expo^{\gamma y},y)$; Figure \ref{f2} shows evaluations for  $y\mapsto v^{(\lambda)}(G_{\lambda}(y),y)$ (black circles), with $G_{\lambda}$ as in \eqref{eq10},  and $y\mapsto v^{(\lambda)}(\expo^{\gamma y}G_{\lambda}(y),y)$ (black triangles).
		\begin{figure}[h!]
			\centering
			\includegraphics[scale=0.7]{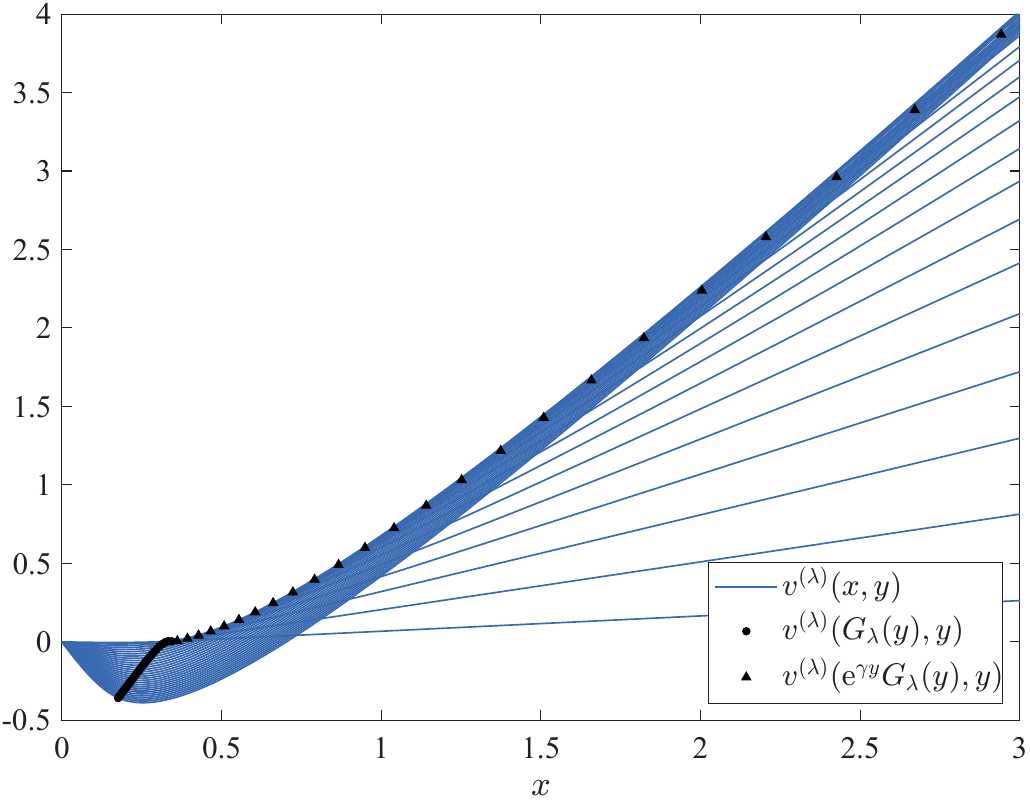}
			\caption{}\label{f2}
		\end{figure}
		
		In figure \ref{f3} depicts the optimal regions which the agent makes the corresponding  decisions when the possibility of default is taking into account.  The plot shows the different boundaries of the set $\partial\mathcal{W}^{(\lambda)}\cup\partial\mathcal{S}_{1}^{(\lambda)}$ (blue lines), with 
		\begin{align*}
			\partial\mathcal{W}^{(\lambda)}\cup\partial\mathcal{S}_{1}^{(\lambda)}&=\{(x,y,w)\in\R_{+}\times\R_{+}\times(-\infty,b): y>0\ \text{and}\ [x=G_{\lambda}(y)\ \text{or}\ x=\expo^{\gamma y}G_{\lambda}(y)] \},
		\end{align*}
		for the default parameter $\lambda=2^{-N}$ with $N\in[-20,8]$ and a step rate size of 0.5. Additionally, the figure shows $\partial\mathcal{W}\cup\partial\mathcal{S}_{1}$ (red line), and $\partial\mathcal{W}^{(\infty)}\cup\partial\mathcal{S}_{1}^{(\infty)}$ (green line), with 
		\begin{align*}
			\partial\mathcal{W}\cup\partial\mathcal{S}_{1}&=\{(x,y,w)\in\R_{+}\times\R_{+}\times[b,\infty): y>0\ \text{and}\ [x=F_{0}\ \text{or}\ x=F_{0}\expo^{\gamma y}] \},\\
			\partial\mathcal{W}^{(\infty)}\cup\partial\mathcal{S}^{(\lambda)}_{1}&=\{(x,y,w)\in\R_{+}\times\R_{+}\times[b,\infty): y>0\ \text{and}\ [x=G_{\infty}(y)\ \text{or}\ x=\expo^{\gamma y}G_{\infty}(y)] \},
		\end{align*}
		where $G_{\infty}(y)\eqdef\frac{ C^{\sell}}{1+K[\gamma y+1]}$.
		
		\begin{figure}[h!]
			\centering
			\includegraphics[scale=0.7]{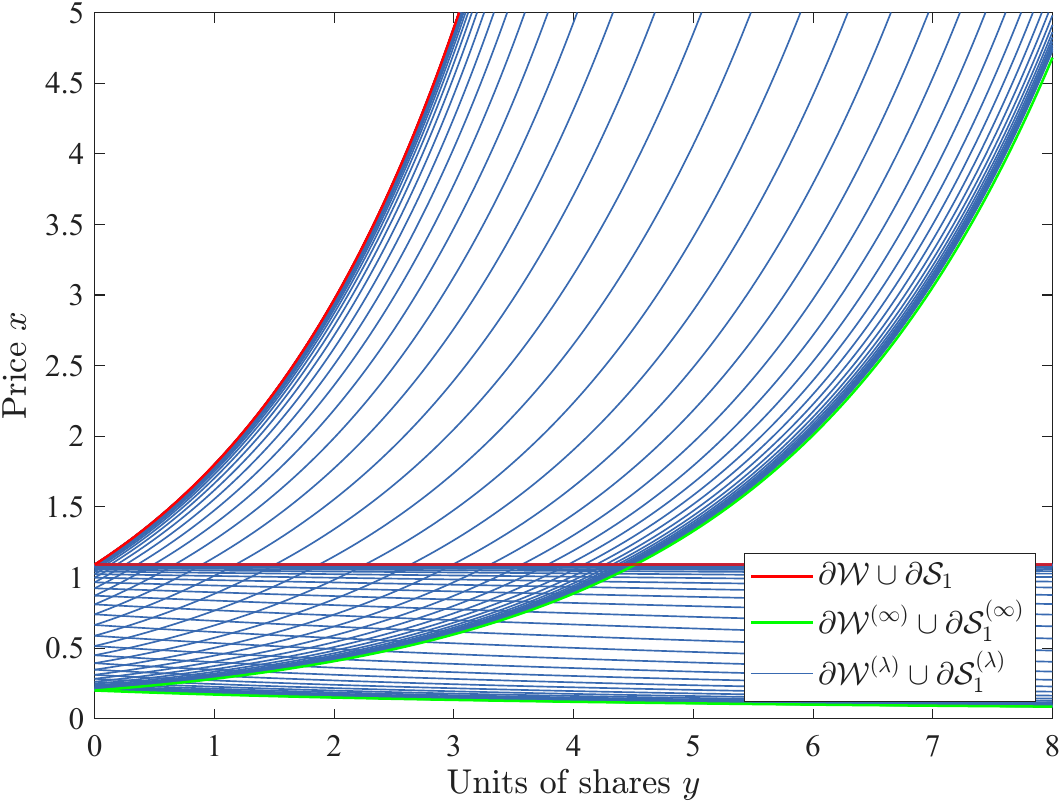}
			\caption{}\label{f3}
		\end{figure}
		In figure \ref{f3} depicts the optimal regions which the agent makes the corresponding  decisions when the possibility of default is taking into account.  The plot shows the different boundaries of the set $\partial\mathcal{W}^{(\lambda)}\cup\partial\mathcal{S}_{1}^{(\lambda)}$ (blue lines), with 
		\begin{align*}
			\partial\mathcal{W}^{(\lambda)}\cup\partial\mathcal{S}_{1}^{(\lambda)}&=\{(x,y,w)\in\R_{+}\times\R_{+}\times(-\infty,b): y>0\ \text{and}\ [x=G_{\lambda}(y)\ \text{or}\ x=\expo^{\gamma y}G_{\lambda}(y)] \},
		\end{align*}
		for the default parameter $\lambda=2^{-N}$ with $N\in[-20,8]$ and a step rate size of 0.5. Additionally, the figure shows $\partial\mathcal{W}\cup\partial\mathcal{S}_{1}$ (red line), and $\partial\mathcal{W}^{(\infty)}\cup\partial\mathcal{S}_{1}^{(\infty)}$ (green line), with 
		\begin{align*}
			\partial\mathcal{W}\cup\partial\mathcal{S}_{1}&=\{(x,y,w)\in\R_{+}\times\R_{+}\times[b,\infty): y>0\ \text{and}\ [x=F_{0}\ \text{or}\ x=F_{0}\expo^{\gamma y}] \},\\
			\partial\mathcal{W}^{(\infty)}\cup\partial\mathcal{S}^{(\lambda)}_{1}&=\{(x,y,w)\in\R_{+}\times\R_{+}\times[b,\infty): y>0\ \text{and}\ [x=G_{\infty}(y)\ \text{or}\ x=\expo^{\gamma y}G_{\infty}(y)] \},
		\end{align*}
		where $G_{\infty}(y)\eqdef\frac{ C^{\sell}}{1+K[\gamma y+1]}$.
		
		\begin{figure}[h!]
			\centering
			\includegraphics[scale=0.7]{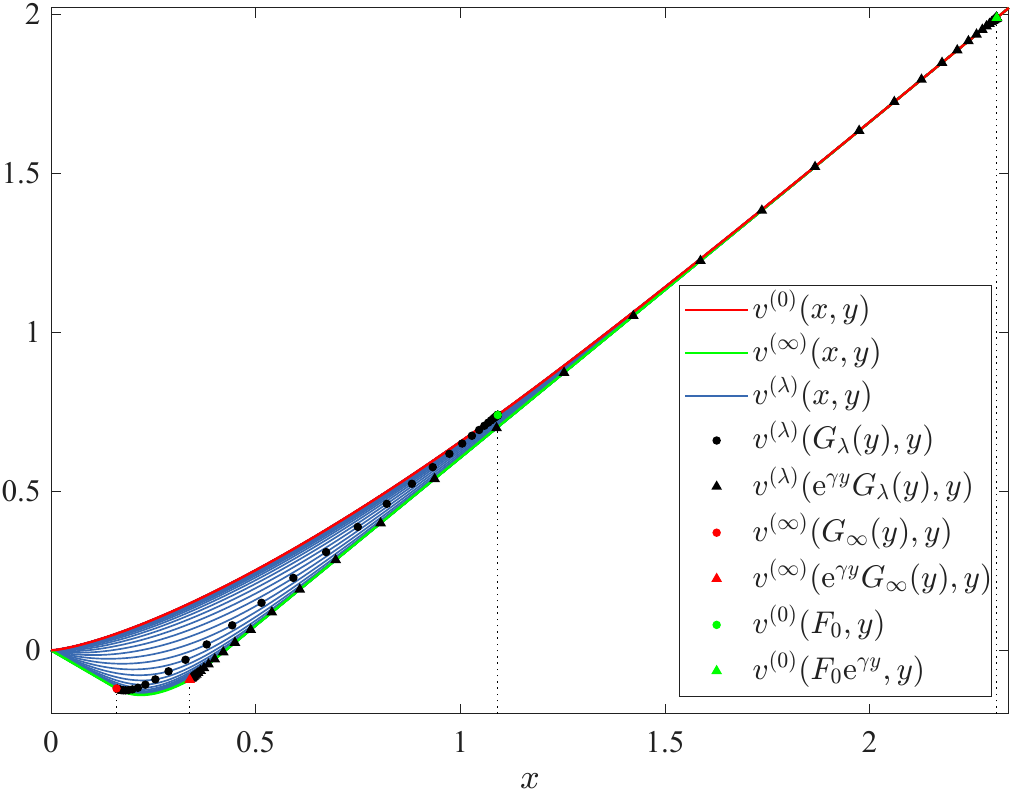}
			\caption{}\label{f4}
		\end{figure}
		To conclude, Figure \ref{f4} shows the graph $x\mapsto v^{(0)}(x,y)$ (red line), $x\mapsto v^{(\lambda)}(x,y)$ (blue line), and $x\mapsto v^{(\infty)}(x,y)$ (green line), for $y=1.5$ and $\lambda=2^{N}$, with $N\in[-20,8]$ and a step rate size of 0.5, where $v^{(\infty)}$ is some limit function.  The figure  also displays pointwise evaluations as functions of $\lambda$:  black circular makers indicate  $\lambda\mapsto v^{(\lambda)}(G_{\lambda}(y),y)$,   and the triangular makers indicate $\lambda\mapsto v^{(\lambda)}(\expo^{\gamma y}G_{\lambda}(y),y)$. Moreover, in Figure \ref{f4}, the red circular and red triangular makers represent  $v^{(\infty)}(G_{\infty}(y),y)$ and $ v^{(\infty)}(\expo^{\gamma y}G_{\infty}(y),y)$. Meanwhile, the green circular and green triangular makers are  $v^{(0)}(F_{0},y)$ and $ v^{(0)}(F_{0}\expo^{\gamma y},y)$.

		\appendix
		\section{Appendix: Proof of Theorem \ref{HJB_proof}}\label{proof_HJB_app}
		%\section{Solution to the HJB equation}
		In this appendix we present the construction of  an explicit  solution for  the HJB equation described in Theorem    \ref{HJB_proof}. It is done  in several steps, as it requires that some technical results are justified.\\
		\textit{Step(i).-Construction of the candidate solution.}\\
		Let $(x,y,w)\in\mathbb{W}$ with $w\geq b$. Then, the first equation in \eqref{HJB1} reduces to
		\begin{align}\label{HJB3.1}
			%\eqdef \frac{1}{2}\sigma^2x^2f_{xx}(x,y)+\mu f_x(x,y)-\left(\lambda \psi(b-b_0)+\delta\right)f(x,y),\\
			\frac{1}{2}\sigma^2x^2 v_{xx}(x,y,w)+\mu x  v_x(x,y,w)+\frac{1}{2}v_{ww}+\rho\sigma x v_{xw}-\delta v(x,y,w)=0.
		\end{align}
		We seek solutions of the form $v(x,y,w)=A(y,w)x^{n_0}$, where $n_0>1$ is the positive root of
		\begin{equation}\label{eq3}
			\frac{1}{2}\sigma^2 l[l-1]+\mu l-\delta=0.
		\end{equation}
		Substituting this ansatz into \eqref{HJB3.1} and using \eqref{eq3}, we obtain
		%\begin{align*}
		%	x^{n_0}A(y)C(w)\frac{1}{2}\sigma^2\left(\frac{1}{2}\sigma^2 n_0(n_0-1)+\mu n_0-\delta\right)+A(y)x^{n_0}\left(\frac{1}{2}C''(w)+\rho\sigma C'(w)\right)=0.
		%\end{align*}
		%\red{
			%	\begin{multline*}
				%		x^{n_0}A(y,w)\frac{1}{2}\sigma^2\left(\frac{1}{2}\sigma^2 n_0(n_0-1)+\mu n_0-\delta\right)\\
				%		+x^{n_0}\left(\frac{1}{2}A_{ww}(y,w)+\rho\sigma n_{0} A_{w}(y,w)\right)=0.
				%	\end{multline*}%}
		%	The form of the previous equation suggests looking at solutions $A$ that satisfy
		\[
		\frac{1}{2}A_{ww}(y,w)+\rho\sigma n_0 A_{w}(y,w)=0,
		\]
		%}
	%\[
	%C(w)=\left(-\frac{1}{2\rho\sigma}\expo^{-2\rho\sigma w}+C\right),
	%\]
	%\red{
		whose general solution is
		\[
		A(y,w)=-\frac{1}{2\rho\sigma n_{0}}\expo^{-2\rho\sigma n_{0} w}\alpha(y)+\beta(y),
		\]%}
	for suitable functions $\alpha,\beta:\R\mapsto\R$. 
	%is a solution to 
	%\[
	%\frac{1}{2}C''(w)+\rho\sigma C'(w)=0.
	%\]
	%\red{
		Thus, a candidate solution to \eqref{HJB1} in this section of $\mathbb{W}$ is
		%\begin{align*}
		%	v(x,y,w)=A(y)x^{n_0}\left(-\frac{1}{2\rho\sigma}\expo^{-2\rho\sigma w}+C\right).
		%\end{align*}
		%\red{
			\begin{align*}
				v(x,y,w)=x^{n_0}\left[-\frac{1}{2\rho\sigma n_{0}}\expo^{-2\rho\sigma n_{0} w}\alpha(y)+\beta(y)\right].
			\end{align*}
			%}
		In the selling region $\mathbb{S}$, the second equation in \eqref{HJB1} takes the form
		\begin{align}\label{exp_0}
			-\gamma xv_x(x,y,w)-v_y(x,y,w)+x- C^{\sell}=0\quad \text{for}\ (x,y,w)\in\mathcal{S},
		\end{align}
		which implies
		\begin{align*}
			-\gamma xv_{xx}(x,y,w)-\gamma v_x(x,y,w)-v_{yx}(x,y,w)+1=0\quad \text{for}\ (x,y,w)\in\mathcal{S}.
		\end{align*}
		We look now for functions $A$ and $F$ such that the associated solution $v$ is twice continuously differentiable in the $x$-variable and continuously differentiable in the $y$-variable.
		These requirements lead to the following system of equations:
		\begin{align*}
			&-\gamma n_0A(y,w)x^{n_0}-A_{y}(y,w)x^{n_0}+x- C^{\sell}\big|_{x=F(y,w)}=0,\notag\\
			&-\gamma n^2_0A(y,w)x^{n_0-1}-n_0A_{y}(y,w)x^{n_0-1}+1\big|_{x=F(y,w)}=0.
		\end{align*}
		Solving this system yields
		\[
		F(y,w)=\frac{n_0 C^{\sell}}{n_0-1}=: F_0,
		\]
		and
		\[
		A(y,w)=\beta(y)=\frac{1}{\gamma n_0^2}\left[\frac{n_0-1}{n_0 C^{\sell}}\right]^{n_0-1}[1-\expo^{-\gamma n_0y}],
		\]
		where the latter follows from the constraint $A(0,w)=0$ for all $w\in\R$.
		Consequently, for $(x,y,w)\in \mathbb{W}$, 
		\[
		v(x,y,w)=\frac{1}{\gamma n_0^2}\left[\frac{n_0-1}{n_0 C^{\sell}}\right]^{n_0-1}[1-\expo^{-\gamma n_0y}]x^{n_0}.
		\]
		
		We now consider the solution in the selling region $\mathbb{W}$. Following Section~5 in Guo and Zervos \cite{ZG15}, for 
		$(x,y,w)\in\mathcal{S}$ with $F_0<x<F_0\expo^{\gamma y}$,
		\begin{align*}
			v(x,y,w)=v(F_0,y-\mathbb{Y}(x,F_0),w)+\frac{1}{\gamma}x\left[1-\expo^{-\gamma \mathbb{Y}(x,F_0)}\right]- C^{\sell}\mathbb{Y}(x,F_0),
		\end{align*}
		where $\mathbb{Y}$ is defined in \eqref{fun_Y}. If $x>F_0\expo^{\gamma y}$, then
		\begin{align*}
			v(x,y,w)=\frac{1}{\gamma}x\left[1-\expo^{-\gamma y}\right]- C^{\sell}y.
		\end{align*}
		The following result summarizes the conclusions above.
		\begin{lemma}
			If $w\geq b$, then $v$ is given by
			\begin{align}\label{sol>b}
				v(x,y,w)=
				\begin{cases}
					\frac{x^{n_0}}{\gamma n_0^2}\left[\frac{n_0-1}{n_0 C^{\sell}}\right]^{n_0-1}[1-\expo^{-\gamma n_0y}]&\text{if}\ 0<x\leq F_{0},\\
					\vspace{-0.4cm}&\\
					v(F_0,y-\mathbb{Y}(x,F_0),w)&\\
					\quad+\frac{x}{\gamma}\left[1-\expo^{-\gamma -\mathbb{Y}(x,F_0)}\right]- C^{\sell}-\mathbb{Y}(x,F_0)&\text{if}\ F_{0} < x\leq F_{0}\expo^{\gamma y},\\
					\vspace{-0.4cm}&\\
					\frac{x}{\gamma}\left[1-\expo^{-\gamma y}\right]- C^{\sell}y&\text{if}\ x>F_{0}\expo^{\gamma y}. %,\\
				\end{cases}
			\end{align}
			which is $\mathcal{C}^2$ in $x$ and  $\mathcal{C}^1$ in $y$.
		\end{lemma}
		
		We now turn our attention to the solution of \eqref{HJB1} in the region $\mathcal{W}^{(\lambda)}$, when $w<b$, where the HJB equation becomes
		%\begin{align*}
		%	\frac{1}{2}\sigma^2x^2 v_{xx}(x,y,w)+\mu x  v_x(x,y,w)+\frac{1}{2}v_{ww}+\rho\sigma x v_{xw}-(\lambda+\delta) v(x,y,w)=-\lambda Kxy.
		%\end{align*}
		%\red{
			\begin{equation*}
				\frac{1}{2}\sigma^2x^2 v_{xx}(x,y,w)+\mu x  v_x(x,y,w)+\frac{1}{2}v_{ww}+\rho\sigma x v_{xw}-[\lambda+\delta] v(x,y,w)=\lambda Kxy.
			\end{equation*}%}
		By Assumption \eqref{cond_drift}, $\delta+\lambda>\mu$, a solution to this equation is given by
		%\[
		%v(x,y,w)=B(y)x^{n_1}C(w)-\frac{\lambda K}{\mu-\delta-\lambda}xy,
		%\]
		%\red{
			\[
			v(x,y,w)=x^{n_1}B(y,w)-\frac{\lambda K}{\delta+\lambda-\mu}xy,
			\]%}
		where $n_1$ is the positive root of 
		\begin{equation}\label{def_n1}
			\frac{1}{2}\sigma^2 l[l-1]+\mu l-[\lambda+\delta]=0.
		\end{equation}
		Moreover $B(y,w)$ is a solution to
		\begin{equation*}
			\frac{1}{2}B_{ww}(y,w)+\rho\sigma n_{1} B_{w}(y,w)=0,
		\end{equation*}
		whose general form is
		\[
		B(y,w)=-\frac{1}{2\rho\sigma n_{1}}\expo^{-2\rho\sigma n_{1} w}\bar{\alpha}(y)+\bar{\beta}(y),
		\]
		We look for functions $B$ and $G_{\lambda}$ such that the associated solution $v$ is twice continuously differentiable in $x$ and continuously differentiable in $y$. This requirement leads to the following system of equations:
		%\begin{align*}
		%	-\gamma n_1B(y)C(w)x^{n_1}+\gamma xy\frac{\lambda K}{\mu-\delta-\lambda}-B'(y)x^{n_1}C(w)+\gamma x\frac{\lambda K}{\mu-\delta-\lambda}+x- C^{\sell}\Big|_{G_{\lambda}(y)}&=0\notag\\
		%	-\gamma n_1^2B(y)C(w)x^{n_1-1}+\gamma y\frac{\lambda K}{\mu-\delta-\lambda}-B'(y)n_1x^{n_1-1}C(w)+\frac{\lambda K}{\mu-\delta-\lambda}+\gamma x\frac{\lambda K}{\mu-\delta-\lambda}\notag\\+1\Big|_{G_{\lambda}(y)}&=0.
		%\end{align*}
		%\red{
			\begin{align*}
				&-\gamma n_1B(y,w)x^{n_1}+\gamma xy\frac{\lambda K}{\delta+\lambda-\mu}-B_{y}(y,w)x^{n_1}+ x\frac{\lambda K}{\delta+\lambda-\mu}+x- C^{\sell}\Big|_{x=G_{\lambda}(y,w)}=0,\notag\\
				&-\gamma n_1^2B(y,w)x^{n_1-1}+\gamma y\frac{\lambda K}{\delta+\lambda-\mu}-B_{y}(y,w)n_1x^{n_1-1}+\frac{\lambda K}{\delta+\lambda-\mu}+1\Big|_{x=G_{\lambda}(y,w)}=0.
			\end{align*}%}
		
		Solving this system yields
		\begin{align}\label{fun_g}
			G_{\lambda}(y)\eqdef G_{\lambda}(y,w)&=\frac{n_1 C^{\sell}}{\displaystyle(n_1-1)\left(1+\frac{\lambda K}{\delta+\lambda-\mu}(\gamma y+1)\right)}
		\end{align}
		which shows in particular that the free boundary does not depend on $w$.
		Finally, using the constraint $B(0,w)=0$ for all $w\in\R$, we obtain
		\[
		B(y)=\bar{\beta}(y)=\expo^{-\gamma n_1y}\int_0^y\frac{ C^{\sell}\expo^{\gamma n_1u}}{(n_1-1)G^{n_1}_{\lambda}(u)}\diff u.
		\]
		
		We now focus on characterizing the solution to \eqref{HJB1} in the selling
		region $\mathcal{S}^{(\lambda)}$. Following again Section 5 in Guo and Zervos\cite{ZG15}, for
		$(x,y,w)\in\mathcal{S}^{(\lambda)}$ with $G_{\lambda}(y)<x<G_{\lambda}(y)\expo^{\gamma y}$, we have
		\begin{align*}%\label{exp_1}
			v(x,y,w)=v(G_{\lambda}(y),y-\mathbb{Y}(x,y),w)+\frac{1}{\gamma}x\left[1-\expo^{-\gamma \mathbb{Y}(x,y)}\right]- C^{\sell}\mathbb{Y}(x,y),
		\end{align*}
		where $\mathbb{Y}$ is defined in \eqref{fun_Y}.
		%\[
		%\mathbb{Y}(x,y)\eqdef\frac{1}{\gamma}\ln\frac{x}{G_{\lambda}(y)}.
		%\]
		
		Additionally, for $(x,y,w)\in\mathcal{S}^{(\lambda)}$ with $x>G_{\lambda}(y)\expo^{\gamma y}$,
		\begin{align*}
			v(x,y,w)=\frac{x}{\gamma}\left[1-\expo^{-\gamma y}\right]- C^{\sell}y.
		\end{align*}
		The following result summarizes the above conclusions.
		\begin{lemma}
			Thus, if $w< b$, then $v$ is given by
			\begin{align}\label{sol<b}
				v(x,y,w)=
				\begin{cases}
					x^{n_1}B(y)-\frac{\lambda K}{\delta+\lambda-\mu}xy&\text{if}\ 0<x\leq G_{\lambda}(y),\\
					\vspace{-0.4cm}&\\
					v(G_{\lambda}(y),y-\mathbb{Y}(x,G_{\lambda}(y)),w)&\\
					\quad+\frac{x}{\gamma}\left[1-\expo^{-\gamma \mathbb{Y}(x,G_{\lambda}(y))}\right]&\\\quad- C^{\sell}\mathbb{Y}(x,G_{\lambda}(y))&\text{if}\ G_{\lambda}(y) < x\leq G_{\lambda}(y)\expo^{\gamma y},\\
					\vspace{-0.4cm}&\\
					\frac{x}{\gamma}\left[1-\expo^{-\gamma y}\right]- C^{\sell}y&\text{if}\ x>G_{\lambda}(y)\expo^{\gamma y}.
				\end{cases}
			\end{align}
			which is $\mathcal{C}^2$ in $x$ and  $\mathcal{C}^1$ in $y$.
		\end{lemma}
		\textit{Step(ii).- Verification that $v$ solves \eqref{HJB1}}\\
		We now verify that the function defined in Proposition~\ref{HJB_proof} (constructed in the previous step) is indeed a solution to the HJB equation given in \eqref{HJB1}.
		First, we note that in the case $w>b$, the function $v$ is given by \eqref{sol>b}. Thus, by the proof of Proposition~5.1 in Guo and Zervos\cite{ZG15}, the function $v$ is a solution of \eqref{HJB1} in the region $\{w\geq b\}$.
		It therefore suffices to verify that $v$ is a solution of \eqref{HJB1} in the case $w< b$. We analyze each of the regions $\mathcal{W}^{(\lambda)}$, $\mathcal{S}_1^{(\lambda)}$, and $\mathcal{S}_2^{(\lambda)}$ separately.

		\bigskip
		\textit{Region $\mathcal{W}^{(\lambda)}$.-}
		By the construction developed in Step~(i), we have that $v$ satisfies \eqref{HJB_0_wait} in the region $\mathcal{W}^{(\lambda)}$, that is,
		\begin{equation}\label{xi_2}
			\Xi^{(2)}_{y}[v]\eqdef\mathcal{L}v(x,y,w)-\lambda v(x,y,w)+Kxy=0.
		\end{equation}
		In order to show that $v$ satisfies \eqref{HJB1} in the region $\mathcal{W}^{(\lambda)}$, it suffices to verify that
		\begin{equation}\label{xi_1}
			\Xi^{(1)}_{y}[v](x)\eqdef-\gamma xv_x(x,y,w)-v_y(x,y,w)+x- C^{\sell}\leq0
		\end{equation}
		holds in $\mathcal{W}^{(\lambda)}$. 
		
		To this end, using the expression for $v$ in the region $w< b$ given in \eqref{sol<b}, we obtain for $x\in(0,G_{\lambda}(y))$,
		%Dado que $v(x,y,w)=x^{n_{1}}B(y)+\frac{\lambda K}{\mu-\delta-\lambda}xy$ sobre $(0,G_{\lambda}(y))$, tenemos que 
		\begin{align*}
			\Xi^{(1)}_{y}[v](x)&=-x^{n_{1}}[\gamma n_{1} B(y)+B'(y)]-\dfrac{\lambda K x}{\mu-\delta-\lambda}[\gamma y +1]+x-C^{\sell}\notag\\
			&=-x^{n_{1}}\dfrac{C^{\sell}}{[n_{1}-1][G_{\lambda}(y)]^{n_{1}}}-\dfrac{\lambda K x}{\mu-\delta-\lambda}[\gamma y+1]+x-C^{\sell}\notag\\
			&=\dfrac{C^{\sell}}{n_{1}-1}\bigg[-\bigg[\dfrac{x}{G_{\lambda}(y)}\bigg]^{n_{1}}+n_1\dfrac{x}{G_{\lambda}(y)}\bigg]-C^{\sell}.
		\end{align*}
		Since $n_1>1$ (by definition of $n_1$ as the positive root of \eqref{def_n1}), it is straightforward to verify that the mapping
		$x\mapsto \Xi^{(1)}_{y}[v](x)$ is non-decreasing on $(0,G_{\lambda}(y))$. Moreover,
		$\Xi^{(1)}_{y}[v](G_{\lambda}(y))=0$. Hence, \eqref{xi_1} holds in $\mathcal{W}^{(\lambda)}$.
		Together with \eqref{xi_2}, this implies that $v$ solves \eqref{HJB1} in the region $\mathcal{W}^{(\lambda)}$.
		
		\bigskip
		\textit{Region $\mathcal{S}^{(\lambda)}_2$.-}
		We now verify that the function $v$ defined in \eqref{sol<b} solves \eqref{HJB1} in the region $\mathcal{S}^{(\lambda)}_2$. To this end, using \eqref{sol<b}, we obtain the following computation.
		\begin{align*}
			\Xi^{(1)}_{y}[v](x)&=\gamma xv_y(G_{\lambda}(y),y-\mathbb{Y}(x,y),w)\frac{1}{\gamma x}-x\left[1-\expo^{-\gamma \mathbb{Y}(x,y)}\right]\notag\\&-x\left[\expo^{-\gamma \mathbb{Y}(x,y)}\right]+ C^{\sell}-v_x(G_{\lambda}(y),y-\mathbb{Y}(x,y),w)G'(y)\notag\\
			&-v_y(G_{\lambda}(y),y-\mathbb{Y}(x,y),w)\left(1+\frac{G'(y)}{\gamma G_{\lambda}(y)}\right)\notag\\
			&+x\frac{G'(y)}{\gamma G_{\lambda}(y)}- C^{\sell}\frac{G'(y)}{\gamma G_{\lambda}(y)}+x- C^{\sell}\notag\\
			&=\Big(-v_x(G_{\lambda}(y),y-\mathbb{Y}(x,y),w)\gamma G_{\lambda}(y)\notag\\&-v_y(G_{\lambda}(y),y-\mathbb{Y}(x,y),w)+x- C^{\sell}\Big)\frac{G'(y)}{\gamma G_{\lambda}(y)}=0,
		\end{align*}
		where the last equality follows from the fact that $v$ satisfies \eqref{exp_0}, by the smooth fit condition at $x=G_{\lambda}(y)$. Hence, $v$ solves \eqref{HJB_0_sell} in $\mathcal{S}^{(\lambda)}_2$. It just remains to prove that it solves \eqref{HJB_0_wait} in $\mathcal{S}^{(\lambda)}_2$. To this end, note that by \eqref{sol<b}, we obtain for $x\in(G_{\lambda}(y),G_{\lambda}(y)\expo^{\gamma y})$
		\begin{align}\label{fd_s_1}
			v_{x}(x,y,w)&=-\frac{1}{\gamma x}v_y(G_{\lambda}(y),y-\mathbb{Y}(x,y),w)+\frac{1}{\gamma }\left[1-\expo^{-\gamma \mathbb{Y}(x,y)}\right]+\frac{1}{\gamma }\left[\expo^{-\gamma \mathbb{Y}(x,y)}\right]-\frac{1}{\gamma x} C^{\sell}\notag\\
			&=-\frac{1}{\gamma x}[v_y(G_{\lambda}(y),y-\mathbb{Y}(x,y),w)+C^{\sell}]+\frac{1}{\gamma }\notag\\
			&=-\frac{1}{\gamma x}\left[G_{\lambda}(y)^{n_{1}}B'(y-\mathbb{Y}(x,y))+\dfrac{\lambda K}{\mu-\delta-\lambda}G_{\lambda}(y)+C^{\sell}\right]+\frac{1}{\gamma}
		\end{align}
		and noticing that 
		$$ n_{1}B'(y)+\frac{1}{\gamma}B''(y)=-\frac{1}{\gamma}\bigg[\frac{n_{1}C^{\sell}}{n_{1}-1}\bigg]\frac{G'(y)}{G_{\lambda}(y)^{n_{1}+1}}=-\bigg[\frac{\lambda K}{\mu-\delta-\lambda}\bigg]\frac{1}{G_{\lambda}(y)^{n_{1}-1}},$$ 
		it follows that 
		\begin{align}\label{fd_s_2}
			v_{xx}(x,y,w)
			%&=-\frac{1}{\gamma x^{2}}[x\partial_{x}v_y(G_{\lambda}(y),y-\mathbb{Y}(x,y),w)\notag\\&-[v_y(G_{\lambda}(y),y-\mathbb{Y}(x,y),w)+C^{\sell}]]\notag\\
			%&=-\frac{1}{\gamma x^{2}}[-x\mathbb{Y}_{x}(x,y)v_{yy}(G_{\lambda}(y),y-\mathbb{Y}(x,y),w)\notag\\&-[v_y(G_{\lambda}(y),y-\mathbb{Y}(x,y),w)+C^{\sell}]]\notag\\
			&=\frac{1}{\gamma x^{2}}\bigg[\dfrac{1}{\gamma}G_{\lambda}(y)^{n_{1}} B''(y-\mathbb{Y}(x,y))+G_{\lambda}(y)^{n_{1}}B'(y-\mathbb{Y}(x,y))\notag\\&+\dfrac{\lambda K}{\mu-\delta-\lambda}G_{\lambda}(y)+C^{\sell}\bigg]\notag\\
			&=\frac{1}{\gamma x^{2}}\bigg[-[n_{1}-1]G_{\lambda}(y)^{n_{1}}B'(y-\mathbb{Y}(x,y))+C^{\sell}\notag\\
			&-\frac{\lambda K}{\mu-\delta-\lambda}\bigg[\bigg[\frac{G_{\lambda}(y)}{G_{\lambda}(y-\mathbb{Y}(x,y))}\bigg]^{n_{1}}G_{\lambda}(y-\mathbb{Y}(x,y))-G_{\lambda}(y)\bigg]\bigg]\notag\\
			%&=\frac{1}{\gamma x^{2}}\bigg[-[n_{1}-1]G_{\lambda}(y)^{n_{1}}\bigg[-\gamma n_{1} B(y-\mathbb{Y}(x,y))\notag\\&+\bigg[\dfrac{C^{\sell}}{n_{1}-1}\bigg]\frac{1}{G_{\lambda}(y-\mathbb{Y}(x,y))^{n_{1}}}\bigg]+C^{\sell}\notag\\
			%&-\frac{\lambda K}{\mu-\delta-\lambda}\bigg[\bigg[\frac{G_{\lambda}(y)}{G_{\lambda}(y-\mathbb{Y}(x,y))}\bigg]^{n_{1}}G_{\lambda}(y-\mathbb{Y}(x,y))-G_{\lambda}(y)\bigg]\bigg]\notag\\
			&=\frac{1}{\gamma x^{2}}\bigg[\gamma n_{1}[n_{1}-1]G_{\lambda}(y)^{n_{1}} B(y-\mathbb{Y}(x,y))-\frac{G_{\lambda}(y)^{n_{1}}C^{\sell}}{G_{\lambda}(y-\mathbb{Y}(x,y))^{n_{1}}}+C^{\sell}\notag\\
			&-\frac{\lambda K}{\mu-\delta-\lambda}\bigg[\bigg[\frac{G_{\lambda}(y)}{G_{\lambda}(y-\mathbb{Y}(x,y))}\bigg]^{n_{1}}G_{\lambda}(y-\mathbb{Y}(x,y))-G_{\lambda}(y)\bigg]\bigg].
		\end{align}
		Therefore, using \eqref{fd_s_1} and \eqref{fd_s_2}, we obtain for $x\in(G_{\lambda}(y),G_{\lambda}(y)\expo^{\gamma y})$,
		\begin{align}\label{xi_2_s_2}
			\Xi^{(2)}_{y}[v](x)
			&=\frac{\sigma^{2}}{2\gamma}\bigg[-[n_{1}-1]G_{\lambda}(y)^{n_{1}}\bigg[-\gamma n_{1} B(y-\mathbb{Y}(x,y))\notag\\
			&\quad+\bigg[\dfrac{C^{\sell}}{n_{1}-1}\bigg]\frac{1}{G_{\lambda}(y-\mathbb{Y}(x,y))^{n_{1}}}\bigg]+C^{\sell}\notag\\
			&\quad-\frac{\lambda K}{\mu-\delta-\lambda}\bigg[\bigg[\frac{G_{\lambda}(y)}{G_{\lambda}(y-\mathbb{Y}(x,y))}\bigg]^{n_{1}}G_{\lambda}(y-\mathbb{Y}(x,y))-G_{\lambda}(y)\bigg]\bigg]\notag\\
			&\quad+\frac{\mu}{\gamma}\bigg[-\bigg[G_{\lambda}(y)^{n_{1}}\bigg[-\gamma n_{1} B(y-\mathbb{Y}(x,y))\notag\\
			&\quad+\bigg[\dfrac{C^{\sell}}{n_{1}-1}\bigg]\frac{1}{G_{\lambda}(y-\mathbb{Y}(x,y))^{n_{1}}}\bigg]\notag\\
			&\quad+\dfrac{\lambda K}{\mu-\delta-\lambda}G_{\lambda}(y)+C^{\sell}\bigg]+x\bigg]-(\lambda+\delta) v(x,y,w)-\lambda Kxy\notag\\
			&=-\frac{1}{\gamma}\bigg[\frac{\lambda+\delta}{n_{1}}\bigg]\bigg[\dfrac{C^{\sell}}{n_{1}-1}\bigg]\frac{G_{\lambda}(y)^{n_{1}}}{G_{\lambda}(y-\mathbb{Y}(x,y))^{n_{1}}}\notag\\
			&\quad-\frac{\sigma^{2}}{2\gamma}\frac{\lambda K}{\mu-\delta-\lambda}\bigg[\frac{G_{\lambda}(y)}{G_{\lambda}(y-\mathbb{Y}(x,y))}\bigg]^{n_{1}}G_{\lambda}(y-\mathbb{Y}(x,y))\notag\\
			&\quad-\frac{1}{\gamma}\bigg[\mu-\frac{\sigma^{2}}{2}\bigg]\bigg[\dfrac{\lambda K}{\mu-\delta-\lambda}G_{\lambda}(y)+C^{\sell}\bigg]+x\frac{\mu}{\gamma}\notag\\
			&\quad-(\lambda+\delta) \bigg[\frac{\lambda K}{\mu-\delta-\lambda}G_{\lambda}(y)[y-\mathbb{Y}(x,y)]+\frac{1}{\gamma}\left[x-G_{\lambda}(y)\right]\notag\\&\quad- C^{\sell}\mathbb{Y}(x,y)\bigg]-\lambda Kxy.
		\end{align}
		In order to verify that $v$ satisfies \eqref{HJB1} in $\mathcal{S}^{(\lambda)}_2$, it suffices to show that $\Xi^{(2)}_{y}[v](x)\leq 0$ for $x\in(G_{\lambda}(y),G_{\lambda}(y)\expo^{\gamma y})$. To this end, we show that the function $x\mapsto \Xi^{(2)}_{y}[v](x)$ is non-decreasing for $x\in(G_{\lambda}(y),G_{\lambda}(y)\expo^{\gamma y})$, and that $\Xi^{(2)}_{y}[v](G_{\lambda}(y)+)=0$. We start by proving the latter. Note that by \eqref{xi_2_s_2},
		%\red{Hay que mostrar que es decreciente y que $\Xi^{(2)}_{y}[v](G_{\lambda}(y)+)=0$}
		%
		%\begin{align*}
		%\Xi^{(2)}_{y}[v](G_{\lambda}(y)+)&=-\frac{1}{\gamma}\bigg[\frac{\lambda+\delta}{n_{1}}\bigg]\bigg[\dfrac{C^{\sell}}{n_{1}-1}\bigg]-\frac{\sigma^{2}}{2\gamma}\frac{\lambda K}{\mu-\delta-\lambda}G_{\lambda}(y)\notag\\
		%&\quad-\frac{1}{\gamma}\bigg[\mu-\frac{\sigma^{2}}{2}\bigg]\bigg[\dfrac{\lambda K}{\mu-\delta-\lambda}G_{\lambda}(y)+C^{\sell}\bigg]+G_{\lambda}(y)\frac{\mu}{\gamma}\notag\\
		%&\quad-(\lambda+\delta) \bigg[\frac{\lambda K}{\mu-\delta-\lambda}G_{\lambda}(y)y\bigg]-\lambda KG_{\lambda}(y)y\notag\\
		%&=-\frac{1}{\gamma}\bigg[\frac{\lambda+\delta}{n_{1}}\bigg]\bigg[\dfrac{C^{\sell}}{n_{1}-1}\bigg]-\frac{1}{\gamma}\bigg[\mu-\frac{\sigma^{2}}{2}\bigg] C^{\sell}\notag\\
		%&\quad-\frac{\mu}{\gamma}\frac{\lambda K}{\mu-\delta-\lambda}G_{\lambda}(y)(\gamma y+1)\notag\\
		%&=\quad \frac{ C^{\sell}}{\gamma}\left[\frac{\sigma^2}{2}+\frac{\mu}{n_1-1}-\frac{(\lambda+\delta)}{n_1(n_1-1)}\right]-G_{\lambda}(y)=-G_{\lambda}(y)\leq 0,
		%\end{align*}
		%where in the last equality we used that
		%\begin{align*}
		%-\frac{\mu}{\gamma}\frac{\lambda K}{\mu-\delta-\lambda}G_{\lambda}(y)(\gamma y+1)(n_1-1)+G_{\lambda}(y)(n_1-1)=n_1 C^{\sell}.
		%\end{align*}
		\begin{align}\label{Xi2_G}
			\Xi^{(2)}_{y}[v](G_{\lambda}(y)+)
			&=-\frac{1}{\gamma}\bigg[\frac{\lambda+\delta}{n_{1}}\bigg]\bigg[\dfrac{C^{\sell}}{n_{1}-1}\bigg]-\frac{\sigma^{2}}{2\gamma}\frac{\lambda K}{\mu-\delta-\lambda}G_{\lambda}(y)\notag\\
			&\quad-\frac{1}{\gamma}\bigg[\mu-\frac{\sigma^{2}}{2}\bigg]\bigg[\dfrac{\lambda K}{\mu-\delta-\lambda}G_{\lambda}(y)+C^{\sell}\bigg]+G_{\lambda}(y)\frac{\mu}{\gamma}\notag\\
			&\quad-(\lambda+\delta) \bigg[\frac{\lambda K}{\mu-\delta-\lambda}G_{\lambda}(y)y\bigg]-\lambda KG_{\lambda}(y)y\notag\\
			%&=-\frac{1}{\gamma}\bigg[\frac{\lambda+\delta}{n_{1}}\bigg]\bigg[\dfrac{C^{\sell}}{n_{1}-1}\bigg]-\frac{1}{\gamma}\bigg[\mu-\frac{\sigma^{2}}{2}\bigg]C^{\sell}\notag\\
			%&-\frac{\sigma^{2}}{2\gamma}\frac{\lambda K}{\mu-\delta-\lambda}G_{\lambda}(y)-\frac{1}{\gamma}\bigg[\mu-\frac{\sigma^{2}}{2}\bigg]\dfrac{\lambda K}{\mu-\delta-\lambda}G_{\lambda}(y)\notag\\&+G_{\lambda}(y)\frac{\mu}{\gamma}\notag\\
			%&-(\lambda+\delta) \bigg[\frac{\lambda K}{\mu-\delta-\lambda}G_{\lambda}(y)y\bigg]-\lambda KG_{\lambda}(y)y\notag\\
			%&=-\frac{C^{\sell}}{\gamma}\bigg\{\bigg[\frac{\lambda+\delta}{n_{1}}\bigg]\bigg[\dfrac{1}{n_{1}-1}\bigg]+\bigg[\mu-\frac{\sigma^{2}}{2}\bigg]\bigg\}\notag\\
			%&+G_{\lambda}(y)\frac{\mu}{\gamma}\bigg[1-\frac{\lambda K}{\mu-\delta-\lambda}\bigg]-\frac{\lambda \mu K}{\mu-\delta-\lambda} G_{\lambda}(y)y\notag\\
			&=-\frac{C^{\sell}}{\gamma}\bigg\{\bigg[\frac{\lambda+\delta}{n_{1}}\bigg]\bigg[\dfrac{1}{n_{1}-1}\bigg]+\bigg[\mu-\frac{\sigma^{2}}{2}\bigg]\bigg\}\notag\\
			&\quad+G_{\lambda}(y)\frac{\mu}{\gamma}\bigg\{1-\frac{\lambda K}{\mu-\delta-\lambda}[1+\gamma y]\bigg\}\notag\\
			%&=-\frac{C^{\sell}}{\gamma}\bigg\{\bigg[\frac{\lambda+\delta}{n_{1}}\bigg]\bigg[\dfrac{1}{n_{1}-1}\bigg]+\bigg[\mu-\frac{\sigma^{2}}{2}\bigg]\bigg\}+\frac{\mu}{\gamma}\bigg[\frac{n_{1}C^{\sell}}{n_{1}-1}\bigg]\notag\\
			&=-\frac{C^{\sell}}{\gamma}\bigg\{\bigg[\frac{\lambda+\delta}{n_{1}}\bigg]\bigg[\dfrac{1}{n_{1}-1}\bigg]+\bigg[\mu-\frac{\sigma^{2}}{2}\bigg]-\mu\bigg[\frac{n_{1}}{n_{1}-1}\bigg]\bigg\}=0.
		\end{align}
		We now prove that the mapping $x \mapsto \Xi^{(2)}_{y}[v](x)$ is non-decreasing for $x \in (G_{\lambda}(y), G_{\lambda}(y)\expo^{\gamma y})$. 
		To this end, note that, by \eqref{fun_g}, the following relation holds:
		\begin{align}\label{ser_G}
			G_{\lambda}'(y)=G^2_{\lambda}(y)\frac{n_1-1}{n_1 C^{\sell}}\frac{\lambda K}{\mu-\delta-\lambda}\gamma.
		\end{align}
		Hence, differentiating \eqref{xi_2_s_2} and applying \eqref{ser_G} yields
		\begin{align}\label{der_Xi_2}
			\Xi^{(2)\prime}_{y}[v](x)&=\frac{\lambda K}{\mu-\delta-\lambda}\left(\frac{G_{\lambda}(y)}{G_{\lambda}(y-\mathbb{Y}(x,y))}\right)^{n_1}\frac{G_{\lambda}(y-\mathbb{Y}(x,y))}{n_1 x\gamma}\Bigg[-(\lambda+\delta)\notag\\&-\frac{\sigma^2}{2}\frac{(n_1-1)^2}{ C^{\sell}}\frac{\lambda K}{\mu-\delta-\lambda}G_{\lambda}(y-\mathbb{Y}(x,y))\Bigg]+\frac{\mu}{\gamma}\notag\\
			&-\frac{(\lambda+\delta)}{\gamma}\Bigg[-\frac{\lambda K}{\mu-\delta-\lambda}\frac{G_{\lambda}(y)}{x}+1-\frac{ C^{\sell}}{x}\Bigg]-\lambda Ky\notag\\
			%&=\frac{\lambda K G_{\lambda}(y)^{n_{1}}}{n_{1}\gamma[\delta+\lambda-\mu]}\left(\frac{1}{xG_{\lambda}(y-\mathbb{Y}(x,y))^{n_{1}-1}}\right)\Bigg[(\lambda+\delta)\notag\\
			%&-\frac{[\sigma[n_1-1]]^{2}\lambda K}{2C^{\sell}[\delta+\lambda-\mu]}G_{\lambda}(y-\mathbb{Y}(x,y))\Bigg]\notag\\
			%&-\bigg\{-\frac{\mu}{\gamma}+\frac{(\lambda+\delta)}{\gamma}\Bigg[\frac{\lambda K}{\delta+\lambda-\mu}\frac{G_{\lambda}(y)}{x}-\frac{ C^{\sell}}{x}+1\Bigg]+\lambda Ky\bigg\}\notag\\
			&=\left(\frac{1}{xG_{\lambda}(y-\mathbb{Y}(x,y))^{n_{1}-1}}\right)\bigg\{\frac{\lambda K G_{\lambda}(y)^{n_{1}}}{n_{1}\gamma[\delta+\lambda-\mu]}\Bigg[(\lambda+\delta)\notag\\&-\frac{[\sigma[n_1-1]]^{2}\lambda K}{2C^{\sell}[\delta+\lambda-\mu]}G_{\lambda}(y-\mathbb{Y}(x,y))\Bigg]\notag\\
			&-G_{\lambda}(y-\mathbb{Y}(x,y))^{n_{1}-1}\bigg\{\frac{(\lambda+\delta)}{\gamma}\Bigg[\frac{\lambda K}{\delta+\lambda-\mu}G_{\lambda}(y)- C^{\sell}\Bigg]\notag\\&+x\bigg[\frac{\lambda+\delta-\mu}{\gamma}+\lambda Ky\bigg]\bigg\}\bigg\}.
		\end{align}
		Note that the mapping 
		\begin{equation*}
			x\mapsto (\lambda+\delta)-\frac{[\sigma[n_1-1]]^{2}\lambda K}{2C^{\sell}[\delta+\lambda-\mu]}G_{\lambda}(y-\mathbb{Y}(x,y))
		\end{equation*}
		is strictly decreasing in $x$, while the mapping
		\begin{multline*}
			x\mapsto G_{\lambda}(y-\mathbb{Y}(x,y))^{n_{1}-1}\bigg\{\frac{(\lambda+\delta)}{\gamma}\Bigg[\frac{\lambda K}{\delta+\lambda-\mu}G_{\lambda}(y)- C^{\sell}\Bigg]+x\bigg[\frac{\lambda+\delta-\mu}{\gamma}+\lambda Ky\bigg]\bigg\}
		\end{multline*}
		is strictly increasing. Combining these observations in \eqref{der_Xi_2} shows that the overall mapping \(x\mapsto \Xi^{(2)\prime}_{y}[v](x)\) is non-increasing. Therefore, to verify \(\Xi^{(2)\prime}_{y}[v](x)\le 0\) for \(x\in(G_{\lambda}(y),G_{\lambda}(y)\expo^{\gamma y})\), it suffices to check that \(\Xi^{(2)\prime}_{y}[v](G_{\lambda}(y))\le 0\) for \(y>0\). To this end, note that using \eqref{der_Xi_2}, we can write
		\begin{align}\label{Xi_2<0}
			\Xi^{(2)\prime}_{y}[v](G_{\lambda}(y))
			&=\frac{1}{G_{\lambda}(y)^{n_{1}}}\bigg\{\frac{\lambda K G_{\lambda}(y)^{n_{1}}}{n_{1}\gamma[\delta+\lambda-\mu]}\Bigg[(\lambda+\delta)-\frac{[\sigma[n_1-1]]^{2}\lambda K}{2C^{\sell}[\delta+\lambda-\mu]}G_{\lambda}(y)\Bigg]\notag\\
			&\quad-G_{\lambda}(y)^{n_{1}-1}\bigg\{\frac{(\lambda+\delta)}{\gamma}\Bigg[\frac{\lambda K}{\delta+\lambda-\mu}G_{\lambda}(y)- C^{\sell}\Bigg]\notag\\
			&\quad+G_{\lambda}(y)\bigg[\frac{\lambda+\delta-\mu}{\gamma}+\lambda Ky\bigg]\bigg\}\bigg\}=\frac{1}{G_{\lambda}(y)}H(y),
		\end{align}
		where
		\begin{align*}
			H(y)%&\eqdef\frac{\lambda K}{n_1\gamma(\lambda+\delta-\mu)}G_{\lambda}(y)\left[(\lambda+\delta)-\frac{\sigma^2}{2}\frac{(n_1-1)^2}{ C^{\sell}(\lambda+\delta-\mu)}\lambda KG_{\lambda}(y)\right]\\
			%&-\bigg[\frac{(\lambda+\delta)}{\gamma}\left(\frac{\lambda K}{\lambda+\delta-\mu}G_{\lambda}(y)- C^{\sell}\right)\notag\\&+G_{\lambda}(y)\left(\frac{\lambda+\delta-\mu}{\gamma}+\lambda K y\right)\bigg]\\
			&\eqdef\frac{\lambda K}{n_1\gamma(\lambda+\delta-\mu)}G_{\lambda}(y)\left[(\lambda+\delta)-\frac{\sigma^2}{2}\frac{(n_1-1)^2}{ C^{\sell}(\lambda+\delta-\mu)}\lambda KG_{\lambda}(y)\right]\\
			&\quad-\bigg[\frac{(\lambda+\delta)}{\gamma}\left(\frac{\lambda K}{\lambda+\delta-\mu}G_{\lambda}(y)- C^{\sell}\right)\notag\\&\quad+G_{\lambda}(y)\frac{\lambda+\delta-\mu}{\gamma}\left(1+\frac{\lambda K y\gamma}{\lambda+\delta-\mu}\right)\bigg].
		\end{align*}
		Using \eqref{fun_g} we obtain that
		\[
		G_{\lambda}(y)\left[1+\frac{\lambda K y\gamma}{\lambda+\delta-\mu}\right]=\frac{n_1}{n_1-1} C^{\sell}-\frac{\lambda K}{\lambda+\delta-\mu}G_{\lambda}(y).
		\]
		Hence, for $y>0$,
		\begin{align}\label{fun_H<0}
			H(y)&=\frac{\lambda K(\lambda+\delta)}{n_1\gamma(\lambda+\delta-\mu)}G_{\lambda}(y)\notag\\
			&\quad-\frac{\lambda K}{n_1\gamma(\lambda+\delta-\mu)}G_{\lambda}(y)\left[\frac{\sigma^2}{2}\frac{(n_1-1)^2}{ C^{\sell}(\lambda+\delta-\mu)}\lambda KG_{\lambda}(y)\right]\notag\\
			&\quad-\frac{(\lambda+\delta)}{\gamma}\left(\frac{\lambda K}{\lambda+\delta-\mu}G_{\lambda}(y)- C^{\sell}\right)\notag\\
			&\quad-\frac{\lambda+\delta-\mu}{\gamma}\left(\frac{n_1}{n_1-1} C^{\sell}-\frac{\lambda K}{\lambda+\delta-\mu}G_{\lambda}(y)\right)\notag\\
			%&=-\frac{\lambda K}{n_1\gamma(\lambda+\delta-\mu)}G_{\lambda}(y)\left[\frac{\sigma^2}{2}\frac{(n_1-1)^2}{ C^{\sell}(\lambda+\delta-\mu)}\lambda KG_{\lambda}(y)\right]\notag\\
			%&+\frac{\lambda K(\lambda+\delta)}{n_1\gamma(\lambda+\delta-\mu)}G_{\lambda}(y)-\frac{(\lambda+\delta)}{\gamma}\left(\frac{\lambda K}{\lambda+\delta-\mu}G_{\lambda}(y)\right)\notag\\&+\frac{\lambda+\delta-\mu}{\gamma}\left(\frac{\lambda K}{\lambda+\delta-\mu}G_{\lambda}(y)\right)\notag\\
			%	&+\frac{(\lambda+\delta)}{\gamma} C^{\sell}-\frac{\lambda+\delta-\mu}{\gamma}\left(\frac{n_1}{n_1-1}\right)  C^{\sell}\notag\\
			%&=-\frac{\lambda K}{n_1\gamma(\lambda+\delta-\mu)}G_{\lambda}(y)\left[\frac{\sigma^2}{2}\frac{(n_1-1)^2}{ C^{\sell}(\lambda+\delta-\mu)}\lambda KG_{\lambda}(y)\right]\notag\\&+\bigg\{(\lambda+\delta)-\mu n_{1}\bigg\}\frac{G_{\lambda}(y)\lambda K}{\gamma n_{1}[\lambda+\delta-\mu]}\notag\\
			%	&+\frac{(\lambda+\delta)}{\gamma} C^{\sell}-\frac{\lambda+\delta-\mu}{\gamma}\left(\frac{n_1}{n_1-1}\right)  C^{\sell}\notag\\
			&=\frac{\lambda K}{n_1\gamma(\lambda+\delta-\mu)}G_{\lambda}(y)\left[(\lambda+\delta)-\frac{\sigma^2}{2}\frac{(n_1-1)^2}{ C^{\sell}(\lambda+\delta-\mu)}\lambda KG_{\lambda}(y)\right]\notag\\
			&\quad-\left[\frac{\mu\lambda K }{\gamma(\lambda+\delta-\mu)}G_{\lambda}(y)+\frac{ C^{\sell}}{\gamma}\frac{\lambda+\delta-n_1\mu}{n_1-1}\right]\notag\\
			&=-\frac{\lambda K}{n_1\gamma(\lambda+\delta-\mu)}G_{\lambda}(y)\left[\frac{\sigma^2}{2}\frac{(n_1-1)^2}{ C^{\sell}(\lambda+\delta-\mu)}\lambda KG_{\lambda}(y)\right]\notag\\
			&\quad+\frac{\lambda+\delta-n_1\mu}{\gamma n_1}\left(\frac{\lambda K}{\lambda+\delta-\mu}G_{\lambda}(y)-\frac{n_1 C^{\sell}}{n_1-1}\right)\notag\\
			%&=-\frac{\lambda K}{n_1\gamma(\lambda+\delta-\mu)}G_{\lambda}(y)\left[\frac{\sigma^2}{2}\frac{(n_1-1)^2}{ C^{\sell}(\lambda+\delta-\mu)}\lambda KG_{\lambda}(y)\right]\notag\\
			%	&-\frac{\lambda+\delta-n_1\mu}{\gamma n_1}G_{\lambda}(y)\left[1+\frac{\lambda K y\gamma}{\lambda+\delta-\mu}\right]\notag\\
			&=-\frac{\lambda K}{n_1\gamma(\lambda+\delta-\mu)}G_{\lambda}(y)\left[\frac{\sigma^2}{2}\frac{(n_1-1)^2}{ C^{\sell}(\lambda+\delta-\mu)}\lambda KG_{\lambda}(y)\right]\notag\\
			&\quad-\frac{\sigma^2}{2\gamma}(n_1-1)G_{\lambda}(y)\left[1+\frac{\lambda K y\gamma}{\lambda+\delta-\mu}\right]\leq 0.
		\end{align}
		Therefore, using \eqref{fun_H<0} in \eqref{Xi_2<0} shows that $\Xi^{(2)\prime}_{y}[v](G_{\lambda}(y)) \le 0$ for $y>0$. 
		It then follows that $\Xi^{(2)\prime}_{y}[v](x) \le 0$ for $x \in (G_{\lambda}(y), G_{\lambda}(y)\expo^{\gamma y})$ and $y>0$. 
		By the discussion preceding \eqref{Xi2_G} and identity \eqref{Xi2_G}, we conclude that 
		\[
		\Xi^{(2)}_{y}[v](x) \le 0 \quad \text{for } x \in (G_{\lambda}(y), G_{\lambda}(y)\expo^{\gamma y}) \text{ and } y>0.
		\] 
		Hence, $v$ is a solution to \eqref{HJB1} in $\mathcal{S}^{(\lambda)}_2$.
		
		\bigskip
		\textit{Region $\mathcal{S}^{(\lambda)}_1$.-}
		Note that by \eqref{sol<b}, for $x > G_{\lambda}(y)\expo^{\gamma y}$ and $y \ge 0$, we have
		\begin{align}\label{Xi_1<0reg3}
			\Xi^{(1)}_{y}[v](x)=-x(1-e^{-\gamma y})-xe^{-\gamma y}+C^{\sell}+x-C^{\sell}=0.
		\end{align}
		On the other hand, \eqref{sol<b} gives
		\begin{align*}
			\Xi^{(2)}_{y}[v](x)
			&=(\lambda +\delta)C^{\sell}y-x\bigg[\frac{\lambda+\delta-\mu}{\gamma}(1-\expo^{-\gamma y})+\lambda K y\bigg].\notag\\
			%&\leq \frac{(\mu-(\lambda+\delta))}{\gamma}G_{\lambda}(y)\expo^{\gamma y}(1-\expo^{-\gamma y})+(\lambda+\delta) C^{\sell}y-\lambda K G_{\lambda}(y)\expo^{\gamma y}y\notag\\
			%&\leq \frac{(\mu-(\lambda+\delta))}{\gamma}G_{\lambda}(y)(\expo^{\gamma y}-1)+(\lambda+\delta) C^{\sell}y-\lambda K G_{\lambda}(y)\expo^{\gamma y}y\notag\\
			%&\eqdef\Theta (y).
		\end{align*}
		Observe that for each fixed $y>0$, the function $x \mapsto \Xi^{(2)}_{y}[v](x)$ is decreasing on $(G_{\lambda}(y)\expo^{\gamma y},\infty)$. 
		Hence, it suffices to verify that
		\begin{align}\label{eq1}
			\Xi^{(2)}_{y}[v]&(G_{\lambda}(y)\expo^{\gamma y}+)\notag\\
			&=(\lambda +\delta)C^{\sell}y-G_{\lambda}(y)\expo^{\gamma y}\bigg[\frac{\lambda+\delta-\mu}{\gamma}(1-\expo^{-\gamma y})+\lambda K y\bigg]\notag\\
			&=G_{\lambda}(y)[\Theta^{(1)}(y)-\Theta^{(2)}(y)]<0,
		\end{align}
		where
		\begin{align*}
			\Theta^{(1)}(y)&\eqdef\frac{(\lambda +\delta)C^{\sell}y}{G_{\lambda}(y)}=\frac{[n_1-1][\lambda +\delta]}{n_1}\bigg[y+\frac{\lambda K}{\delta+\lambda-\mu}(\gamma y^{2}+y)\bigg],\notag\\
			\Theta^{(2)}(y)&\eqdef\expo^{\gamma y}\bigg[\frac{\lambda+\delta-\mu}{\gamma}(1-\expo^{-\gamma y})+\lambda K y\bigg].
		\end{align*}
		Since $\Theta^{(1)}$ and $\Theta^{(2)}$ have polynomial and exponential growth, respectively, and $\Theta^{(1)}(0+) = \Theta^{(2)}(0+) = 0$, it suffices to check that
		\begin{align}\label{eq2}
			\Theta^{(1)\prime}(0+)<\Theta^{(2)\prime}(0+). 
		\end{align}
		Computing the derivatives and taking $y \downarrow 0$ yields
		\[
		\Theta^{(1)\prime}(0+) = \frac{[n_1-1][\lambda + \delta][\lambda + \delta - \mu + \lambda K]}{n_1[\lambda + \delta - \mu]}, \quad
		\Theta^{(2)\prime}(0+) = \lambda + \delta - \mu + \lambda K,
		\]
		which shows that \eqref{eq2} holds. Therefore, \eqref{eq1} is satisfied for every $y>0$, and we conclude that
		\[
		\Xi^{(2)}_{y}[v](x) \le 0 \quad \text{for } x \in (G_{\lambda}(y)\expo^{\gamma y}, \infty) \text{ and } y>0.
		\]
		Combining this with \eqref{Xi_1<0reg3}, we obtain that $v$ is a solution to \eqref{HJB1} in $\mathcal{S}^{(\lambda)}_1$.


\begin{thebibliography}{99}
	\bibitem{AFS} \sc Alfonsi, A., Fruth, A., Schied, A. \rm Constrained portfolio liquidation in a limit book order, {\it Advances in Mathematics of Finance, Banach Center Publ. 83}, Polish Academy of Sciences, Warsaw, (2008), 9--25.
	
	\bibitem{AFS2} \sc Alfonsi, A., Fruth, A., Schied, A. \rm Optimal execution strategies in limit order books with general shape functions, {\it Quant. Finance} {\bf 10}, (2010), 143--157.
	
	\bibitem{A-C-2} \sc Almgren, R., Chriss, N. \rm Value under liquidation, {\it Risk} {\bf 12}, (1999), 61--63.
	
	\bibitem{A-C-3} \sc Almgren, R., Chriss, N. \rm Optimal execution of portfolio transactions, {\it J.  Risk} {\bf 3}, (2000), 5-39.
	
	
	
	
	
	
	\bibitem{BL} \sc Bertsimas, D., Lo, W., \rm Optimal control of execution costs, {\it J. Risk} {\bf 3}, (2000), 5--39.
	
	
	\bibitem{BY} \sc Bremaud, P., and M. Yor \rm Changes of Filtration and of Probability Measures. {\it Z.f.W.} {\bf 45}, 269–-295, (1978).
	
	%\bibitem{DI1993} {\sc Dupuis, P. and Ishii, H.} {\rm S{DE}s with oblique reflection on nonsmooth domains.} {\it Ann. Probab.} {\bf 21} (1993), 554--580. 
	
	
	\bibitem{E-J-Y}  \sc  Eliiott, R., Jeanblanc, M.  and M. Yor \rm On models of default risk. {\it Math. Finance} {\bf 10} (2000), 179--195.
	\bibitem{EEH} \sc Ettinger, B., Evans, S.N. and Hening, A., \rm Killed Brownian motion with a prescribed  lifetime distribution and models for default. {\it Ann. Applied Probab.} {24} (2014), 1-33. 
	
	
	\bibitem{EF} \sc Engle, R., Ferstenberg, R., \rm Execution risk, {\it J. Portfolio Manage.} {\bf 33}, (2007), 34--44.
	
	\bibitem{GS} \sc Gatheral, J., Schied, A. \rm Dynamical models of market impact and algorithms for order execution, in J,-P. Fouque \& J. Langsam, ends, Handbook on Systemic Risk, Cambridge University Press, (2013),  579-602.
	
	\bibitem{GSS} \sc Gatheral, J., Schied, A., Slynko, A. \rm Transient linear price impact and Fredholm integral equations, {\it Math. Finance} {\bf  22}, (2012), 445-474.
	
	%\bibitem{GL} \sc Guo, X., Liu, J. \rm Optimal stopping at the maximum of Geometric Brownian motion when signals are received, {\it J. Appl. Prob.}  {\bf 42}, (2005),  826-838.
	
	\bibitem{ZG15} \sc Guo, X., Zervos, M. \rm Optimal execution with multiplicative price impact, {\it SIAM J. Financial Math.} {\bf 6}, (2015), 281-305.
	
	\bibitem{HM} \sc He, H., Mamayski, H. \rm Dynamic trading with price impact, {\it J. Econom. Dynam. Control} {\bf  29}, (2005), 891-930.
	
	\bibitem{HMP19} \sc Hern\'andez-Hern\'andez, D, Moreno-Franco, H. and P\'erez J.L., \rm  Periodic strategies in optimal execution with multiplicative price impact, {\t   Mathematical Finance} {\bf 29}, (2019),  1039-1065.
	
	
	\bibitem{HS} \sc Huberman, G., Stanzi, W. \rm Price manipulation and quasi-arbitrage, {\it Econometrica} {\bf  72}, (2004), 1247-1275.
	
	\bibitem{L} \sc Lokka, A. \rm Optimal liquidation in a limit order book for a risk-averse investor, {\it Math. Finance} {\bf 24}, (2014), 696-727.
	
	\bibitem{OW} \sc Obizhaeva, A., Wang, J. \rm Optimal trading strategy and supply/demand dynamics, {\it J. Financ. Market} {\bf  16}, (2013), 1--13.
	
	\bibitem{PSS} \sc Predoiu, S., Shaiket, G., Shreve, S. \rm Optimal execution in a general one-sided limit order book, {\it SIAM J. Financial Math.} {\bf 2}, (2011), 183-212.
	
	\bibitem{QJL2023}  {\sc  Qiu, M.,  Jin, Z.,  Li, S.}  {\rm Optimal risk sharing and dividend strategies under default contagion: A semi-analytical approach}, {\it Insurance Math. Econom.}  {\bf 113}, (2023).
	
	
	\bibitem{SS} \sc Schied, A., Sch\"oneborn, T. \rm Risk aversion and the dynamics of optimal liquidation strategies in illiquid markets, {\it Finance Stoch.} {\bf 13}, (2009), 181--204.
	
	\bibitem{SST} \sc Schied, A., Sch\"oneborn, T., Tehranchi, M. \rm Optimal basket liquidation for CARA investors is deterministic, {\it Appl. Math. Finance} {\bf 17}, (2010), 471--489.
	
	
\end{thebibliography}
\end{document}